\documentclass[12pt,oneside]{amsart}

\usepackage{amsmath,amssymb,amsthm,xcolor,tikz-cd}
\usepackage[a4paper, total={5.5in, 9in}]{geometry}

\usepackage{stmaryrd}

\usepackage{enumitem}

\usepackage[backref]{hyperref}

\theoremstyle{plain}
\newtheorem{theorem}{Theorem}[section]
\newtheorem{corollary}[theorem]{Corollary}
\newtheorem{proposition}[theorem]{Proposition}
\newtheorem{lemma}[theorem]{Lemma}

\newtheorem{definition}[theorem]{Definition} 

\theoremstyle{definition}
\newtheorem{example}[theorem]{Example}
\newtheorem{remark}[theorem]{Remark}
\newtheorem{assumption}[theorem]{Assumption}

\newcommand{\Bl}{\mathrm{Bl}}

\newcommand{\Aut}{\mathrm{Aut}}
\newcommand{\Tr}{\mathrm{Tr}}
\newcommand{\Hess}{\mathrm{Hess}}
\newcommand{\Ric}{\mathrm{Ric}}
\renewcommand{\d}{\partial}
\newcommand{\db}{\overline{\partial}}
\newcommand{\ext}{\mathrm{ext}}
\renewcommand{\O}{\mathrm{O}}

\newcommand{\A}{\mathcal{A}}

\newcommand{\X}{\mathcal{X}}

\newcommand{\DF}{\mathrm{DF}}
\newcommand{\E}{\mathcal{E}}
\newcommand{\Ch}{\mathrm{Ch}}

\begin{document}
	
\title{Stability of weighted extremal manifolds through blowups}
\author{Michael Hallam}
\address{Department of Mathematics, Aarhus University, Ny Munkegade 118, DK-8000 Aarhus C, Denmark.}
\email{hallam@math.au.dk}
\maketitle

\begin{abstract}
	In \cite{Hal23}, we showed that the blowup of a weighted extremal K{\"a}hler manifold at a relatively stable fixed point admits a weighted extremal metric. Using this result, we prove that a weighted extremal manifold is relatively weighted K-polystable. In particular, a weighted cscK manifold is weighted K-polystable. This strengthens both the weighted K-semistability proved by Lahdili \cite{Lah19} and Inoue \cite{Ino20}, and the weighted K-polystability with respect to smooth degenerations by Apostolov--Jubert--Lahdili \cite{AJL21}, allowing for possibly singular degenerations.
\end{abstract}

\tableofcontents

\section{Introduction}

The question of when a K{\"a}hler manifold admits a canonical K{\"a}hler metric is a cornerstone of K{\"a}hler geometry, and has yielded fascinating links with algebraic geometry. The problem was first posed by Calabi in the case of K{\"a}hler--Einstein metrics \cite{Cal54}, and later in greater generality for constant scalar curvature K{\"a}hler (cscK) metrics and extremal K{\"a}hler metrics \cite{Cal82}.

The existence of a cscK or extremal metric is conjecturally equivalent to various notions of K-stability, through the Yau--Tian--Donaldson conjecture \cite{Yau93, Tia97, Don02}. The earliest forms of K-stability were introduced in \cite{Tia97,Don02}; also relevant to us will be the condition of relative K-stability \cite{Sze07}, as well as the generalisations of K-stability to non-algebraic K{\"a}hler manifolds \cite{SD18, DR17, SD20, Der18}.

By now there are various approaches to proving K-stability of manifolds with canonical K{\"a}hler metrics, such as employing geodesic rays and asymptotic slope formulae of energy functionals; see \cite{PRS08, BHJ19, BDL20, SD18, SD20}. Although some of these methods will be important in this work, we instead focus on proofs which make use of existence results for canonical K{\"a}hler metrics on blowups. 

The first result of this kind was proven for cscK manifolds with discrete automorphisms \cite{AP06}, and was later generalised to extremal manifolds in \cite{APS11, Sze12, Sze15, DS21}. Using such existence results, one can ``upgrade" K-semistability to K-stability. This argument was first carried out by Stoppa in the polarised cscK case \cite{Sto09, Sto11}, and later by Stoppa--Sz{\'e}kelyhidi in the polarised extremal case \cite{SS11}; see also \cite{DR17, Der18} for the non-polarised K{\"a}hler setting.

Alongside cscK and extremal K{\"a}hler metrics, many other important notions of canonical K{\"a}hler metric have been introduced. Examples include K{\"a}hler--Ricci solitons, extremal Sasaki metrics \cite{BGS08}, as well as conformally K{\"a}hler Einstein--Maxwell metrics \cite{LeB10}. It was shown by Lahdili that these all fit under the umbrella of \emph{weighted extremal metrics} \cite{Lah19}. Independently, Inoue introduced an important subclass of Lahdili's metrics called \emph{\(\mu\)-cscK metrics}, which provide a convenient generalisation of both extremal metrics and K{\"a}hler--Ricci solitons \cite{Ino22}.

To give a brief description of weighted cscK and extremal metrics, let \((X,\omega)\) be a compact K{\"a}hler manifold, and let \(T\) be a compact real torus acting on \(X\) by hamiltonian isometries. Let \(\mu:X\to\mathfrak{t}^*\) be a moment map for the action, and write \(P:=\mu(X)\) for the moment polytope. The \emph{weight functions} that give rise to the weighted extremal equation are positive smooth functions \(v,w:P\to\mathbb{R}_{>0}\). Using these, one defines a \emph{weighted scalar curvature} \(S_{v,w}(\omega)\), which is a deformation of the usual scalar curvature \(S(\omega)\); we refer to Section \ref{subsec:weighted_extremal_metrics} for the precise definitions and examples. If \(S_{v,w}(\omega)\) is constant we say \(\omega\) is \emph{weighted cscK}, and if \(\nabla^{1,0}S_{v,w}(\omega)\) is a holomorphic vector field then we call \(\omega\) a \emph{weighted extremal metric}. All of the canonical K{\"a}hler metrics above can be recovered as weighted cscK or extremal metrics.

Alongside weighted cscK metrics, Lahdili introduced a notion of weighted K-stability with respect to smooth degenerations \cite{Lah19}. The definition involves integrating the weighted scalar curvature of a K{\"a}hler metric over the total space of a test configuration, which requires the test configuration be smooth. While this is assumption is sufficient to recover a good notion of weighted K-semistability, for weighted K-(poly)stability one must also define stability with respect to test configurations that have singular total space. Such a definition was achieved by Inoue \cite{Ino20}, who used the theory of equivariant cohomology to define the weighted Donaldson--Futaki invariant of a singular test configuration.

One slight restriction we make in our work, as was made in \cite{Ino20}, is that we assume the weight functions \(v\) and \(w\) can be written as the composition of a linear functional \(\ell_\xi:P\to\mathbb{R}\) with real analytic functions \(f_v\) and \(f_w\) on \(\ell_\xi(P)\subset\mathbb{R}\). In practice, most examples of weighted cscK and extremal metrics satisfy this condition. The assumption on the weights is laid out clearly in Assumption \ref{ass:weights}.

We now state the main results. Let \((X,\alpha_T)\) be a compact K{\"a}hler manifold equipped with a hamiltonian isometric \(T\)-action; here \(\alpha\) is a fixed K{\"a}hler class on \(X\) and \(\alpha_T\) a fixed extension of \(\alpha\) to a \(T\)-equivariant degree 2 cohomology class (such an extension is equivalent to  a choice of normalisation for the moment map of the \(T\)-action). In \cite{Hal23}, we showed that if \(X\) admits a weighted extremal metric, so too does its blowup at a relatively stable point that is fixed by both the \(T\)-action and the extremal field. Using this, we will show that a weighted extremal manifold is relatively weighted K-stable.
	
\begin{theorem}\label{thm:main}
	Let \((X,\alpha_T)\) be a \((v,w)\)-extremal manifold, where \(T\) is maximal in the hamiltonian isometry group of the weighted extremal metric, and the weight functions \(v,w\) satisfy Assumption \ref{ass:weights}. Then \((X,\alpha_T)\) is relatively weighted K-polystable.
\end{theorem}

From this, we will easily obtain the following corollary:

\begin{corollary}\label{cor:main}
	Let \((X,\alpha_T)\) be a \((v,w)\)-cscK manifold, where \(T\) is maximal and the weight functions \(v,w\) satisfy Assumption \ref{ass:weights}. Then \((X,\alpha_T)\) is weighted K-polystable.
\end{corollary}

These results rely on the already proven fact that a weighted extremal manifold is relatively weighted K-semistable. This was first shown with respect to smooth degenerations by Lahdili using asymptotic slope techniques \cite{Lah19}, and then extended to singular degenerations by Inoue \cite{Ino20}, who used a continuity argument to reduce to Lahdili's result. Then, Apostolov--Jubert--Lahdili proved that weighted cscK manifolds are weighted K-polystable with respect to smooth degenerations \cite{AJL21}; Corollary \ref{cor:main} strengthens this stability result to possibly singular test configurations, using a completely different approach.

For the proof of the main result, we follow the basic strategy of previous works \cite{Sto09,SS11,Der18}. Namely, we first use that a weighted extremal manifold is relatively weighted K-semistable. We then take a test configuration with vanishing relative Donaldson--Futaki invariant, and seek to show it is a product test configuration; see Section \ref{sec:background} for the relevant definitions here. Assuming it is not a product, we show that there exists a \(T\)-invariant point \(p\in X\) such that \(\Bl_pX\) has a destabilising test configuration, which is obtained from the original test configuration through a standard process. This contradicts the existence result in \cite{Hal23} for weighted extremal metrics on blowups, and hence implies the test configuration must have been a product.

One novelty in our approach is that we compute the change in weighted Donaldson--Futaki invariant of a twisted test configuration using equivariant localisation techniques, in place of asymptotic slope formulae for the weighted Mabuchi functional used in \cite{Der18}. The use of equivariant localisation techniques has become increasingly prominent in K-stability, in particular through the work of Legendre \cite{Leg21} and Inoue \cite{Ino20}. Whereas Legendre's work computes the Donaldson--Futaki invariant directly via equivariant localisation on the total space of the test configuration, Inoue's approach instead pushes forward the equivariant classes to \(\mathbb{C}\) first, and then applies equivariant localisation. This latter approach was first seen in the work of Wang in the unweighted setting \cite{Wan12}, and is this method that we use here.

Along the way, we also correct some minor inaccuracies in \cite{Der18}. In particular, we introduce a new variant of the Chow weight that is invariant under twisting the test configuration by the torus action. In the polarised unweighted case, this invariant is nothing new as it can be obtained by twisting the test configuration by a rational one-parameter subgroup so that it is orthogonal to the torus. However, when the K{\"a}hler class is not rational, or in the presence of weight functions, this new Chow weight is an essential ingredient in expanding the Donaldson--Futaki invariant orthogonal to the torus. To prove the existence of a point with positive \(T\)-invariant Chow weight, we prove a kind of uniform Chow stability, namely that one can find a point with positive Chow weight bounded below by a positive constant times the norm of the test configuration. This eventually allows one to reduce to the rational case via an approximation argument.

\renewcommand{\abstractname}{Acknowledgements}
\begin{abstract}
	I thank Eiji Inoue and Abdellah Lahdili for kindly answering some questions on their works. I thank Ruadha{\'i} Dervan for comments on the paper and discussions on the existence of Chow stable points, and Zakarias Sj{\"o}str{\"o}m Dyrefelt for comments on the paper. I also thank Eveline Legendre and Lars Sektnan for their interest in the project.
\end{abstract}
	
\section{Background}\label{sec:background}
	
\subsection{Weighted extremal metrics}\label{subsec:weighted_extremal_metrics}
		
Here we very briefly review the definitions and our conventions on weighted cscK metrics; a more thorough treatment was already given in \cite{Hal23}. Let \((X,\omega)\) be a compact \(n\)-dimensional K{\"a}hler manifold. The Ricci curvature of \(\omega\) is \[\Ric(\omega):=-\frac{i}{2\pi}\d\db\log\omega^n,\] and the scalar curvature is \[S(\omega):=\Lambda_\omega\Ric(\omega)=\frac{n\,\Ric(\omega)\wedge\omega^{n-1}}{\omega^n}.\]  

Take: \begin{enumerate}
	\item \(T\) a real torus acting effectively on \((X,\omega)\) by hamiltonian isometries,
	\item \(\mu:X\to \mathfrak{t}^*\) a moment map for the \(T\)-action,
	\item \(P:=\mu(X)\subset\mathfrak{t}^*\) the moment polytope \cite{Ati82,GS82},
	\item \(v,w:X\to\mathbb{R}_{>0}\) positive smooth functions.
\end{enumerate} Our sign convention for moment maps is \[\langle d\mu,\xi\rangle = -\omega(\xi,-)\] for all \(\xi\in\mathfrak{t}\), where the pairing \(\langle-,-\rangle\) is the natural one on \(\mathfrak{t}^*\otimes\mathfrak{t}\), and we abuse notation by writing \(\xi\) for the vector field it generates on \(X\). 

\begin{definition}
	Given the above data, we define the \emph{\(v\)-weighted scalar curvature} of \(\omega\) to be \[S_v(\omega):=v(\mu)S(\omega)-2\Delta(v(\mu))+\frac{1}{2}\Tr(g\circ\Hess(v)(\mu)).\] Here  \(S(\omega)=\Lambda_\omega\Ric(\omega)\) is the scalar curvature, \(\Delta=-\d^*\d\) is the K{\"a}hler Laplacian of \(\omega\), and \(g\) is the Riemannian metric determined by \(\omega\).
\end{definition} 

Concretely, the term \(\Tr(g\circ\Hess(v)(\mu))\) may be written \[\sum_{a,b}v_{,ab}(\mu)g(\xi_a,\xi_b),\] where \(\xi_1,\ldots,\xi_r\) is a basis of \(\mathfrak{t}\), and \(v_{,ab}\) denotes the \(ab\)-partial derivative of \(v\) with respect to the dual basis of \(\mathfrak{t}^*\).

\begin{definition}[{\cite{Lah19}}]\label{def:weighted_extremal}
	The metric \(\omega\) is: \begin{enumerate}
		\item a \emph{\((v,w)\)-weighted cscK metric}, if \[S_v(\omega)=c_{v,w}w(\mu),\] where \(c_{v,w}\) is a constant,
		\item a \emph{\((v,w)\)-weighted extremal metric} if the function \[S_{v,w}(\omega):=S_v(\omega)/w(\mu)\] is a holomorphy potential with respect to \(\omega\).
	\end{enumerate}  
\end{definition}

Sometimes we shorten the full name to just a \((v,w)\)-cscK metric, or a \((v,w)\)-extremal metric. If the weight functions \(v\) and \(w\) are understood or irrelevant, we may also refer to such a metric simply as a weighted cscK metric, or a weighted extremal metric.

For \(\xi\in\mathfrak{t}\), we write \(\mu^\xi:=\langle\mu,\xi\rangle=\ell_\xi\circ\mu\), where \(\ell_\xi\in(\mathfrak{t}^*)^*\) is the element corresponding to \(\xi\in\mathfrak{t}\). Given a basis \(\{\xi_a\}\) for \(\mathfrak{t}\), we also write \(\mu^a\) in place of \(\mu^{\xi_a}\). The function \(\mu^\xi\) is then a hamiltonian for the infinitesimal action of \(\xi\) on \(X\). 

We are interested in finding a weighted extremal metric in the class \(\alpha:=[\omega]\). Given a \(T\)-invariant K{\"a}hler potential \(\varphi\in\mathcal{H}^T\), let \[\mu_\varphi:=\mu+d^c\varphi.\] That is, for any \(\xi\in\mathfrak{t}\), \[\mu_\varphi^\xi:=\mu^\xi+d^c\varphi(\xi),\] where we again abuse notation by writing \(\xi\) for the vector field it generates on \(X\). Our convention here is \(d^c:=\frac{i}{2}(\db-\d)\), so that \(dd^c=i\d\db\).

\begin{lemma}[{\cite[Lemma 1]{Lah19}}]\label{lem:moment_map_change}
	With the above definition, \(\mu_\varphi\) is a moment map for the \(T\)-action with respect to \(\omega_\varphi\), and \(\mu_\varphi(X)=P\), where \(P\) is the moment polytope for \(\mu=\mu_0\).
\end{lemma}

With this lemma in mind, it then makes sense to search for a \((v,w)\)-extremal metric in the class \(\alpha\), which is \(T\)-invariant by fiat.

\begin{lemma}[{\cite[Lemma 2]{Lah19}}]\label{lem:const_cvw}
	With the above choice of moment map \(\mu_\varphi\), the following quantities are independent of the choice of \(\varphi\in\mathcal{H}^T\): \begin{enumerate}
		\item \(\int_Xv(\mu_\varphi)\,\omega_\varphi^n\),
		\item \(\int_Xv(\mu_\varphi)\,\Ric(\omega_\varphi)\wedge\omega_\varphi^{n-1}+\int_X\langle dv(\mu_\varphi),-\Delta_{\varphi}\mu_\varphi\rangle\,\omega_\varphi^n\),
		\item \(\int_X S_v(\omega_\varphi)\,\omega_\varphi^n\).
	\end{enumerate} It follows that the constant \(c_{v,w}\) of Definition \ref{def:weighted_extremal} is fixed, given by \[c_{v,w}=\frac{\int_XS_v(\omega)\,\omega^n}{\int_Xw(\mu)\,\omega^n}.\]
\end{lemma} 

We will henceforth write \begin{equation}\label{eq:average_weighted_scalar_curvature}
	\hat{S}_{v,w}:=\frac{\int_X S_{v,w}(\omega) w(\mu) \omega^n}{\int_X w(\mu) \omega^n}
\end{equation} in place of \(c_{v,w}\), to indicate that this is the average of the weighted scalar curvature \(S_{v,w}(\omega)\) with respect to the measure \(w(\mu)\omega^n\).

\begin{remark}\label{rem:Ricci_moment_map}
	The significance of \(-\Delta_\varphi\mu_\varphi\) in \emph{(2)} of Lemma \ref{lem:const_cvw} is that it is a moment map for the Ricci curvature \(\Ric(\omega_\varphi)\), see \cite[Lemma 5]{Lah19}. That is, for any \(\xi\in\mathfrak{t}\), \[\langle d(-\Delta_\varphi\mu_\varphi)(x),\xi\rangle = \Ric(\omega_\varphi)(x)(-,\xi_x).\]
\end{remark}

\begin{example}\label{ex:weighted_cscK_metrics}
	Fix an element \(\xi\in\mathfrak{t}\), and denote by \(\ell_\xi:\mathfrak{t}^*\to\mathbb{R}\) the corresponding element of \((\mathfrak{t}^*)^*\). Let \(a\) be a constant such that \(a+\ell_\xi>0\) on \(P\). Many standard canonical metrics can be obtained from certain choices of the functions \(v,w\) \cite[Section 3]{Lah19}: \begin{enumerate}
		\item {\bf CscK:} Taking \(v\) and \(w\) constant, the weighted cscK equation reduces to \(S(\omega)=const\).
		\item {\bf Extremal:} Taking \(v\) and \(w\) constant again, a weighted extremal metric is precisely an extremal metric in the usual sense, meaning \(\db\nabla^{1,0}S(\omega)=0\).
		\item {\bf K{\"a}hler--Ricci soliton:} For \(X\) Fano and \(\alpha = c_1(X)\), a weighted extremal metric in \(\alpha\) with weights \(v=w=e^{\ell_\xi}\) and extremal field \(\xi\) is a K{\"a}hler--Ricci soliton.
		\item {\bf Extremal Sasaki:} Suppose that \([\omega]\) is the first Chern class \(c_1(L)\) of an ample line bundle \(L\to X\). A choice of K{\"a}hler metric \(\omega_\varphi\in\alpha\) then corresponds to a Sasaki metric on the unit circle bundle \(S\) of \(L^*\). Letting \[v:=(a+\ell_\xi)^{-n-1},\quad w:=(a+\ell_\xi)^{-n-3},\] a \((v,w)\)-extremal metric on \(X\) then corresponds to an extremal Sasaki metric on \(S\) with Sasaki--Reeb field \(\xi\) \cite{AC21,ACL21}.
		\item {\bf Conformally K{\"a}hler--Einstein Maxwell:} Letting \[v=(a+\ell_\xi)^{-2n+1},\quad w=(a+\ell_\xi)^{-2n-1},\] a \((v,w)\)-cscK metric on \(X\) then corresponds to a conformally K{\"a}hler Einstein--Maxwell metric \cite{Lah20}.
		\item {\bf \(v\)-soliton:} Suppose \(X\) is Fano and \(\alpha=c_1(X)\). Taking \(v\) arbitrary and defining \[w(p)=2v(p)(n+\langle d\log v(p),p\rangle),\] the \((v,w)\)-cscK equation then becomes \[\Ric(\omega)-\omega=i\d\db\log v(\mu),\] which is the \(v\)-soliton equation \cite{HL20}.
		\item {\bf \(\mu\)-cscK:} In \cite{Ino22,Ino20}, Inoue introduced and studied a class of \emph{\(\mu\)-cscK metrics}. These are a special class of weighted extremal metrics, given by the same weight functions \(v=w=e^{\ell_\xi}\) and extremal field \(\xi\) as for K{\"a}hler--Ricci solitons, only one drops the condition of \(X\) being Fano \cite[Section 2.1.6]{Ino22}.
	\end{enumerate}
\end{example}
			
Aside from the general \(v\)-solitons, notice that all of these examples are of a particular form, namely the weight functions are the composition of a linear functional \(\ell_\xi:\mathfrak{t}^*\to\mathbb{R}\) with a real analytic function \(\widetilde{v}\) or \(\widetilde{w}\) on a subset of \(\mathbb{R}\). In order to apply Inoue's equivariant intersection theory for weighted cscK metrics \cite{Ino20}, we will assume that our weight functions take this form in later sections.

Furthermore, suppose we choose another moment map \(\mu'=\mu+\eta\) where \(\eta\in\mathfrak{t}^*\) is a constant. We then define \(v',w':P'\to\mathbb{R}\) by \(v'(p'):=v(p'-\eta)\) and \(w'(p'):=w(p'-\eta)\), where \(P'=P+\eta\) is the moment polytope of \(\mu'\). The class \(\alpha\) then admits a weighted extremal (resp. weighted cscK) metric for the weight functions \(v,w\) and moment map \(\mu\) if and only if it admits a weighted extremal (resp. weighted cscK) metric for the weight functions \(v',w'\) and moment map \(\mu'\). It follows we may freely choose our normalisation of the moment map in what follows. 

\begin{assumption}\label{ass:weights}
	We assume: \begin{enumerate}
	\item The weight functions \(v,w:P\to\mathbb{R}\) are of the form \(v(p)=\widetilde{v}(\langle\xi,p\rangle)\), \(w(p)=\widetilde{w}(\langle\xi,p\rangle)\), where:\vspace{1mm} \begin{enumerate}
		\item \(\xi\in\mathfrak{t}\) is a fixed element of the Lie algebra,\vspace{1mm}
		\item \(\widetilde{v}=f^{(n)}\) and \(\widetilde{w}=g^{(n+1)}\), where \(f\) and \(g\) are real analytic functions defined on a neighbourhood of \(\langle P,\xi\rangle\subset\mathbb{R}\).\vspace{1mm}
	\end{enumerate}
		\item The moment polytope \(P\) is translated so that the closed interval \(\langle P,\xi\rangle\subset\mathbb{R}\) has \(0\) as its midpoint, hence the analytic functions \(f\) and \(g\) have power series expansions about \(0\in\mathbb{R}\) that converge on a neighbourhood of \(\langle P,\xi\rangle\).
	\end{enumerate}
\end{assumption}

As a remark, we will later assume that the torus \(T\) is maximal. This does not impose any condition on the existence of solutions: suppose that \(\omega\) is a weighted cscK metric for the torus \(T\), and let \(T'\supset T\) be a larger torus in the hamiltonian isometry group of \((X,\omega)\). There is a natural projection \(p:(\mathfrak{t}')^*\to\mathfrak{t}^*\), which is compatible with the moment maps \(\mu:X\to\mathfrak{t}^*\) and \(\mu':X\to(\mathfrak{t}')^*\). Defining \(v':= v\circ p\) and \(w':= w\circ p \), we see that \(v'(\mu') = v(\mu)\) and \(w'(\mu') = w(\mu)\). It follows easily from the definitions that a \((v,w)\)-cscK (resp. extremal) metric for \(T\) is a \((v',w')\)-cscK (resp. extremal) metric for \(T'\), and vice versa.
		
\subsection{K{\"a}hler test configurations}
		
We next review the theory of K{\"a}hler test configurations, which was introduced in \cite{SD18, DR17}. We take a particular interest in the equivariant geometry of test configurations, via moment maps. We will assume some familiarity with equivariant cohomology (including Cartan's formulation) \cite[Appendix]{Ino20} as well as K{\"a}hler complex spaces \cite{Fis76} and Bott--Chern cohomology \cite[Section 4.6.1]{BG13}.

Let \((X,\omega)\) be a compact K{\"a}hler manifold, and \(T\) be a compact torus acting on \((X,\omega)\) by hamiltonian isometries. Note the action extends to a \(T^{\mathbb{C}}\) action. We write \(\alpha=[\omega]\) for the K{\"a}hler class and identify \(\mathbb{P}^1=\mathbb{C}\cup\{\infty\}\).
				
\begin{definition}
	A \emph{\(T\)-equivariant K{\"a}hler test configuration} for \((X,\alpha)\) is a pair \((\X,\A)\), where: \begin{enumerate}
		\item \(\X\) is a normal compact K{\"a}hler complex space,
		\item \(\A\in H^{1,1}_{\mathrm{BC}}(\X;\mathbb{R})\) is a K{\"a}hler class on \(\X\),
	\end{enumerate} along with the following data: \begin{enumerate}[label=(\roman*)]
		\item A \(T^{\mathbb{C}}\times\mathbb{C}^*\)-action on \(\X\) which preserves the class \(\A\),
		\item A \(T^{\mathbb{C}}\times\mathbb{C}^*\)-equivariant flat surjection \(\pi:\X\to\mathbb{P}^1\), where \(T^{\mathbb{C}}\) acts trivially on \(\mathbb{P}^1\) and \(\mathbb{C}^*\) acts on \(\mathbb{P}^1\) by scalar multiplication on \(\mathbb{C}\subset\mathbb{P}^1\),
		\item A \(T^{\mathbb{C}}\times\mathbb{C}^*\)-equivariant isomorphism \begin{equation}\label{eq:general_fibre}
		\Psi:(\X,\A)|_{\pi^{-1}(\mathbb{P}^1\backslash\{0\})}\cong(X\times\mathbb{P}^1\backslash\{0\},p^*\alpha),
		\end{equation} where \(p:X\times\mathbb{P}^1\backslash\{0\}\to X\) is the projection, and
		\(T^{\mathbb{C}}\times\mathbb{C}^*\) acts diagonally on \(X\times\mathbb{P}^1\backslash\{0\}\).
	\end{enumerate} Two \(T\)-equivariant K{\"a}hler test configurations are \emph{isomorphic} if they are \(T^{\mathbb{C}}\times\mathbb{C}^*\)-equivariantly isomorphic to each other.
\end{definition}
				
Since the \(T\)-action on \((X,\omega)\) is hamiltonian, it admits a moment map \(\mu:X\to\mathfrak{t}^*\). A choice of moment map determines an extension of \(\alpha\) to an equivariant cohomology class \(\alpha_T:=[\omega+\mu]\in H^2_T(X;\mathbb{R})\), by Cartan's formulation of equivariant cohomology. Given such an extension, we will show that the class \(\A\) of a \(T\)-equivariant K{\"a}hler test configuration \((\X,\A)\) admits a unique extension to an equivariant class \(\A_T\) that is compatible with \(\alpha_T\) under the isomorphism \eqref{eq:general_fibre}; for this reason, we will later write \(T\)-equivariant K{\"a}hler test configurations as \((\X,\A_T)\), and refer to them simply as \emph{test configurations}.
				
Let \(\Omega\) be a \(T\times S^1\)-invariant K{\"a}hler form on \(\X\) representing the class \(\A\). To find \(\A_T\), we will construct a moment map on \(\X\) for the \(T\times S^1\)-action with respect to \(\Omega\). In order to do this, it suffices to construct a hamiltonian function for the infinitesimal action of any one-parameter subgroup \(\beta:\mathbb{C}^*\to T^{\mathbb{C}}\times\mathbb{C}^*\) on \(\X\). Since \(\beta(\mathbb{C}^*)\) preserves the class \(\A\), for all \(t\in\mathbb{R}\) there exists a smooth function \(\varphi_t:\X\to\mathbb{R}\) such that \[\beta(e^{-t})^*\Omega-\Omega=i\d\db\varphi_t,\] by definition of Bott--Chern cohomology.
				
\begin{lemma}\label{lem:hamiltonians} There exist \(T\times S^1\)-invariant smooth functions 	\(\varphi_t:\X\to\mathbb{R}\) such that:
	\begin{enumerate}
		\item For all \(t\in\mathbb{R}\), \[\beta(e^{-t})^*\Omega-\Omega=i\d\db\varphi_t,\]
		\item The \(\varphi_t\) depend smoothly on \(t\),
		\item On the regular locus \(\X_{\mathrm{reg}}\) of \(\X\), the function \(-\dot{\varphi}_0/2\) is hamiltonian for the real holomorphic vector field \(V\) generating the \(\beta(S^1)\)-action: \[\Omega(-,V)=-d\dot{\varphi}_0/2,\]
		\item For all \(t\in\mathbb{R}\), \[\beta(e^{-t})^*\dot{\varphi}_0=\dot{\varphi}_t.\]
	\end{enumerate}
\end{lemma}
				
\begin{proof}
	Let us first prove this assuming \(\X\) is smooth. In this case, the map \(i\d\db:C_0^\infty(\X;\mathbb{R})\to\Omega^{1,1}(\X;\mathbb{R})\) has a \(T\times S^1\)-equivariant left-inverse \(L\), since the Laplacian \(\Delta_\Omega:=\Lambda_{\Omega}\circ i\d\db:C^\infty_0(\X;\mathbb{R})\to C^\infty_0(\X;\mathbb{R})\) is invertible and \(T\times S^1\)-equivariant. We define \[\varphi_t:=L(\beta(e^{-t})^*\Omega-\Omega),\] which is clearly smooth in \(t\) and \(T\times S^1\)-invariant. Since \(\beta\) preserves \(\A\), there exist \(\varphi_t'\) satisfying \(\beta(e^{-t})^*\Omega-\Omega=i\d\db \varphi_t'\). Hence \[\varphi_t=L(i\d\db\varphi'_t)=\varphi'_t-\overline{\varphi}'_t,\] where \(\overline{\varphi}'_t\) is the average of \(\varphi'_t\). It follows that \(\beta(e^{-t})^*\Omega-\Omega=i\d\db \varphi_t\), and so the \(\varphi_t\) are \(T\times S^1\)-invariant and satisfy (1) and (2).
	
	Differentiating the equation \(\beta(e^{-t})^*\Omega-\Omega=i\d\db \varphi_t\) at \(t=0\), we produce \[\mathcal{L}_{\mathcal{J}V}\Omega=i\d\db\dot{\varphi}_0,\] where \(V=V^{1,0}+V^{0,1}\) is the real holomorphic vector field generating the \(\beta(S^1)\)-action and \(\mathcal{J}\) is the almost complex structure of \(\X\). This can be rearranged using Cartan's magic formula to get \[d(\Omega(\mathcal{J}V,-)-d^c\dot{\varphi}_0)=0,\] where \(d^c:=\frac{i}{2}(\db-\d)\). Since the vector field \(V\) is hamiltonian, it has a fixed point. Hence by \cite[Theorem 1]{LS94}, there exists a smooth function \(f:\X\to\mathbb{C}\) such that \(\Omega(V^{1,0},-)=\db f\). Since \(V\) is Killing, \(f\) is real-valued by \cite[p. 64]{Sze14}. It follows that \(\Omega(\mathcal{J}V,-)=2d^cf\), which in turn gives \(\Omega(-,V)=-df\). Hence \(dd^c(2f-\dot{\varphi}_0)=0\), so \(2f\) and \(\dot{\varphi}_0\) differ by a constant. Therefore \(\Omega(-,V)=-d\dot{\varphi}_0/2\) and (3) holds.
	
	Finally to show (4), differentiating \(\beta(e^{-t})^*\Omega-\Omega=i\d\db \varphi_t\) at arbitrary \(t\) yields \[\beta(e^{-t})^*i\d\db\dot{\varphi}_0=i\d\db\dot{\varphi}_t.\] Hence \(\beta(e^{-t})^*\dot{\varphi}_0=\dot{\varphi}_t+C(t)\), where \(C:\mathbb{R}\to\mathbb{R}\) is a smooth function such that \(C(0)=0\). If we replace \(\varphi_t\) with \(\varphi_t+\int_0^tC(s)\,ds\) then each of the conditions (1)--(3) are maintained, and (4) is also satisfied.
				
	In the case where \(\X\) is singular, begin by choosing smooth functions \(\varphi_t\) satisfying \(\beta(e^{-t})^*\Omega-\Omega=i\d\db\varphi_t\), which we do not require to vary smoothly with \(t\). Let \(\widetilde{\X}\) be an equivariant resolution of singularities, and denote by \(\widetilde{\Omega}\) (resp. \(\widetilde{\varphi}_t\)) the pullback of \(\Omega\) (resp. \(\varphi_t\)) to \(\widetilde{\X}\). Choose an invariant K{\"a}hler form \(\eta\) on \(\widetilde{\X}\). By the smooth case, there exist functions \(g_t\) varying smoothly with \(t\) such that \[\beta(e^{-t})^*\eta-\eta=i\d\db g_t,\] and for each \(\epsilon>0\) there are functions \(f_{\epsilon,t}\) varying smoothly with \(t\) satisfying \[\beta(e^{-t})^*(\widetilde{\Omega}+\epsilon\eta)-(\widetilde{\Omega}+\epsilon\eta)=i\d\db f_{\epsilon,t}.\] Note we also have \[\beta(e^{-t})^*(\widetilde{\Omega}+\epsilon\eta)-(\widetilde{\Omega}+\epsilon\eta)=i\d\db(\widetilde{\varphi}_t+\epsilon g_t).\] It follows that \(\widetilde{\varphi}_t+\epsilon g_t=f_{\epsilon,t}-c_{\epsilon,t}\), where \(c_{\epsilon,t}\) is a constant depending on \(\epsilon\) and \(t\). Rearranging we get \(\widetilde{\varphi}_t+c_{\epsilon,t}=f_{\epsilon,t}-\epsilon g_t\). Since the right-hand side depends smoothly on \(t\), we can fix \(\epsilon\) and replace \(\varphi_t\) with \(\varphi_t+c_{\epsilon,t}\) so that the functions \(\varphi_t\) vary smoothly with \(t\). Finally since \(f_{\epsilon,t}\) and \(g_t\) are \(T\times S^1\)-invariant and satisfy (3) and (4) for their respective K{\"a}hler metrics, the \(\varphi_t\) are also \(T\times S^1\)-invariant and satisfy (3) and (4).
\end{proof}	
				
\begin{corollary}\label{cor:moment_map}
	There exists a smooth function \(m:\X\to\mathfrak{t}^*\oplus\mathbb{R}\) that is a \(T\times S^1\)-moment map for \(\Omega\), in the sense that: \begin{enumerate}
		\item \(m\) is \(T\times S^1\)-equivariant, and
		\item For all \(\xi\in\mathfrak{t}\oplus \mathbb{R}\), \(\langle dm,\xi\rangle = \Omega(-,\xi)\) on the regular locus \(\X_{\mathrm{reg}}\) of \(\X\).
	\end{enumerate}
\end{corollary}
				
\begin{proof}
	As the Lie algebra of a compact torus, \(\mathfrak{t}\oplus\mathbb{R}\) has a natural lattice. We therefore choose an integral basis for \(\mathfrak{t}\oplus\mathbb{R}\), and apply the previous lemma to show that each basis vector has a \(T\times S^1\)-invariant hamiltonian function. Collecting these hamiltonian functions in a vector using the dual basis for \(\mathfrak{t}^*\oplus\mathbb{R}^*\), we get the moment map \(m\) which is automatically equivariant since the adjoint action on \(\mathfrak{t}\oplus\mathbb{R}\) is trivial.
\end{proof}

When we wish to refer to the component of \(m\) taking values in \(\mathfrak{t}^*\) we shall write \(m_T\), and similarly write \(m_{S^1}\) for the component in \(\mathbb{R}^*=\mathrm{Lie}(S^1)^*\). We are interested in normalising the moment map \(m_T\) so that it is suitably compatible with the original moment map \(\mu:X\to\mathfrak{t}^*\). This is achievable by the following result, which follows easily from \cite[Lemma 7]{Lah19}:
			
\begin{lemma}\label{lem:moment_map_normalisation}
	Denote by \(m_T:\X\to\mathfrak{t}^*\) the \(T\)-moment map for \(\Omega\). We may normalise \(m_T\) so that for each \(\tau\in\mathbb{P}^1\), \(m_T(\X_\tau)=P:=\mu(X)\). In particular, under the isomorphism \(\Psi\) from equation \eqref{eq:general_fibre}, \[\Psi_*[\Omega+m_T]|_{\pi^{-1}(\mathbb{P}^1\backslash\{0\})}=p^*[\omega+\mu]\in H^2_T(X\times\mathbb{P}^1\backslash\{0\};\mathbb{R}).\]
\end{lemma}

\begin{proof}
	The result \cite[Lemma 7]{Lah19} states the equality \(m_T(\X_\tau)=P:=\mu(X)\) holds for all \(\tau\in\mathbb{C}^*\) for a suitable normalisation of \(m_T\); we note this extends to all \(\tau\in\mathbb{P}^1\) by continuity of \(m_T\) and compactness of \(\X\).\footnote{The result in \cite{Lah19} was stated for smooth test configurations, however the proof works equally well in the singular case.} For the statement on equivariant cohomology, first note for any \(\tau\in\mathbb{P}^1\backslash\{0\}\) the condition \(m_T(\X_\tau)=P:=\mu(X)\) implies the restriction \([\Omega+m_T]|_{\X_\tau}\) maps to \([\omega+\mu]\) under \(\Psi\); see \cite[Lemma 1]{Lah19}. To see this implies an equality of equivariant classes on all of \(X\times\mathbb{P}^1\backslash\{0\}\), we merely note the \(T\)-action on \(\mathbb{P}^1\backslash\{0\}\) is trivial, so there is an equality \[H^2_T(X\times\mathbb{P}^1\backslash\{0\};\mathbb{R})\cong H^2_T(X;\mathbb{R})\otimes H^0(\mathbb{P}^1\backslash\{0\})=H^2_T(X;\mathbb{R}),\] where the first isomorphism follows from the Kunneth theorem since the higher cohomology groups of \(\mathbb{P}^1\backslash\{0\}\) vanish.
\end{proof}
			
We would like to say the moment map \(m\) gives us an extension of \(\A\) to an equivariant class \(\A_{T}:=[\Omega+m_T]\in H^2_{T}(\X;\mathbb{R})\). However, the Cartan model for equivariant cohomology is not well suited to singular spaces. Instead, given an equivariant resolution of singularities \(\widetilde{\X}\to\X\), we pull back the equivariant form \(\Omega+m_T\) to \(\widetilde{\Omega}+\widetilde{m}_T\) on \(\widetilde{\X}\), which determines an equivariant cohomology class \(\widetilde{\A}_{T}:=[\widetilde{\Omega}+\widetilde{m}_T]\in H^2_{T}(\widetilde{\X};\mathbb{R})\) extending the pulled back class \(\widetilde{\A}\). We thus abuse notation by writing \[\mathcal{A}_{T\times S^1}:=[\Omega+m]\in H^2_{T\times S^1}(\X;\mathbb{R}),\quad\quad\mathcal{A}_T:=[\Omega+m_T]\in H^2_{T}(\X;\mathbb{R}),\] even though we only define the classes \[\widetilde{\mathcal{A}}_{T\times S^1}:=[\widetilde{\Omega}+\widetilde{m}]\in H^2_{T\times S^1}(\widetilde{\X};\mathbb{R}),\quad\quad\widetilde{A}_{T}:=[\widetilde{\Omega}+\widetilde{m}_T]\in H^2_{T}(\widetilde{\X};\mathbb{R}).\] In particular, \(T\)-equivariant test configurations will henceforth be denoted \((\X,\A_T)\), and we will often call these simply ``test configurations", taking the \(T\)-action as implicit in the notation.

\begin{remark}
	The class \(\widetilde{\A}_T\) on the resolution is independent of the form \(\Omega\) and moment map \(m\), which follows from the choice of normalisation in Lemma \ref{lem:moment_map_normalisation}.
\end{remark}
	
We will need the following generalisation of \cite[Lemma 7]{Lah19}; the proof follows largely the same approach:

\begin{lemma}\label{lem:independence}
	Let \((\X,\A_T)\) be a test configuration for \((X,\alpha_T)\), with \(\Omega+m_T\) an equivariant representative for \(\mathcal{A}_T\). Let \(h:P\to\mathbb{R}\) be an arbitrary smooth function on the moment polytope. Then for all \(\tau\in\mathbb{P}^1\backslash\{0\}\), \[\int_{\X_\tau}h(m_\tau)\Omega_\tau^n=\int_{\X_1}h(m_1)\Omega_1^n=\int_Xh(\mu)\omega^n,\] where \(\Omega_\tau\) (resp. \(m_\tau\)) denotes the restriction of \(\Omega\) (resp. \(m_T\)) to \(\X_\tau\). Furthermore,  for all \(\tau\in\mathbb{P}^1\backslash\{0\}\) and \(\xi\in\mathfrak{t}\), \[\int_{\X_\tau}S_v(\Omega_\tau)\ell_\xi(m_\tau)\Omega_\tau^n=\int_{\X_1}S_v(\Omega_1)\ell_\xi(m_1)\Omega_1^n=\int_XS_v(\omega)\ell_\xi(\mu)\omega^n.\]
\end{lemma}
	
\begin{proof}
	It suffices to prove these equalities for \(\tau\in\mathbb{C}^*\); the equalities at \(\tau=\infty\) will follow from continuity. Denote by \(\lambda\) the \(\mathbb{C}^*\)-action on the test configuration. For \(\tau\in\mathbb{C}^*\) define \(\omega_\tau:=\lambda(\tau)^*\Omega_\tau\) and \(\mu_\tau:=\lambda(\tau)^*m_\tau\), so that \(\mu_\tau\) is a moment map for \(\omega_\tau\) on \(\X_1\cong X\). Our first goal is then to show that \[\int_{\X_1}h(\mu_\tau)\omega_\tau^n\] is independent of \(\tau\). Let \(t:=-\log|\tau|\); by circle invariance it suffices to show the derivative in \(t\) of the integral vanishes. Denote by \(V\) the infinitesimal generator of the circle action on \(\X\). Then \begin{align*}
		\frac{d}{dt}\mu_\tau^\xi &= \frac{d}{dt}(\lambda(e^{-t})^*m_\tau^\xi) \\
		&=\lambda(e^{-t})^*\mathcal{L}_{\mathcal{J}V}m_T^\xi \\
		&=\lambda(e^{-t})^*dm_T^\xi(\mathcal{J}V) \\
		&=-\lambda(e^{-t})^*(dm_T^\xi,dm_{S^1})_{\Omega} \\
		&=\lambda(e^{-t})^*dm_{S^1}(\mathcal{J}\xi) \\
		&=-\lambda(e^{-t})^*((d_{\mathcal{V}}m_{S^1})^{\#}, \xi)_{\Omega_\tau},
	\end{align*} where \((-,-)_{\Omega}\) is the Hermitian inner product determined by \(\Omega\), \(d_{\mathcal{V}}\) denotes the vertical (fibrewise) derivative, and \(\#\) is conversion of a 1-form to a vector field on \(\X_\tau\) via \(\Omega_\tau\).
	Similarly, using (1), (3) and (4) of Lemma \ref{lem:hamiltonians}, \[\frac{d}{dt}\omega_\tau = -2\lambda(e^{-t})^*(i\d\db m_{S^1}|_{\X_\tau}).\] It follows that \begin{align*}
		\frac{d}{dt}\int_{\X_1}h(\mu_\tau)\omega_\tau^n =&
		-\sum_a\int_{\X_1}h_{,a}(\mu_\tau)\lambda(e^{-t})^*((d_{\mathcal{V}}m_{S^1})^{\#}, \xi_a)_{\Omega_\tau}\lambda(e^{-t})^*\Omega_\tau^n \\
		&-2\int_{\X_1}h(\mu_\tau)\lambda(e^{-t})^*\Delta_{\Omega_\tau}(m_{S^1})\lambda(e^{-t})^*\Omega_\tau^n \\
		=& -\sum_a\int_{\X_\tau}h_{,a}(m_\tau)((d_{\mathcal{V}}m_{S^1})^{\#}, \xi_a)_{\Omega_\tau}\Omega_\tau^n \\
		&+ \sum_a\int_{\X_\tau}h_{,a}(m_\tau)((d_{\mathcal{V}}m_{S^1})^{\#}, \xi_a)_{\Omega_\tau}\Omega_\tau^n \\
		=&0.
	\end{align*} Here \(\{\xi_a\}\) is a basis for \(\mathfrak{t}\), and \(h_{,a}\) denotes the partial derivative of \(h\) in the \(\xi^a\)-direction, where \(\{\xi^a\}\) is the dual basis for \(\mathfrak{t}^*\). The equality \[\int_{\X_1}h(m_1)\Omega_1^n=\int_Xh(\mu)\omega^n\] then follows from the general fact that such integrals are independent of the choice of representative for the equivariant class \([\omega+\mu]\), which can easily be seen by differentiating along a straight line path between two representatives.
	
	We now consider the integrals involving the \(v\)-weighted scalar curvature; our aim is to show that the derivative of \[\int_{\X_1}S_v(\omega_\tau)\ell_\xi(\mu_\tau)\omega_\tau^n\] in \(t\) vanishes. Lahdili computes the derivative of \(S_v(\omega_\tau)\) in \(t\) as \[\frac{d}{dt}S_v(\omega_\tau)=-\mathcal{D}_\tau^*(v(\mu_\tau)\mathcal{D}_\tau\dot{\varphi}_t)+\frac{1}{2}(dS_v(\omega_\tau),d\dot{\varphi}_t),\] where \(\mathcal{D}_\tau:=\db\nabla^{1,0}_{\omega_\tau}\); see \cite[Lemma B.1]{Lah19}. Here \(\mathcal{D}_\tau^*\) refers to the formal \(L^2\)-adjoint with respect to the measure \(\omega^n\). From this we have \begin{align*}
		\frac{d}{dt}\int_{\X_1}S_v(\omega_\tau)\ell_\xi(\mu_\tau)\omega_\tau ^n 
		=&-\int_{\X_1}\mathcal{D}_\tau^*(v(m_\tau)\mathcal{D}_\tau\dot{\varphi}_t)\ell_\xi(\mu_\tau)\omega_\tau^n \\
		&+\frac{1}{2}\int_{\X_1}(dS_v(\omega_\tau),d\dot{\varphi}_t)\ell_\xi(\mu_\tau)\omega_\tau^n \\
		&-\int_{\X_1}S_v(\omega_\tau)\lambda(e^{-t})^*((d_{\mathcal{V}}m_{S^1})^{\#}, \xi))_{\Omega_\tau} \omega_\tau^n \\
		&+\int_{\X_1}S_v(\omega_\tau)\ell_\xi(\mu_\tau)\Delta_{\omega_\tau}(\dot{\varphi}_t)\omega_\tau^n \\
		=&\,\,0.
	\end{align*} Here the first line vanishes since \(\ell_\xi(\mu_\tau)\) is in the kernel of \(\mathcal{D}_\tau\), and the remaining three lines cancel via integration by parts similarly to above using that \(\dot{\varphi}_t=-2\lambda(e^{-t})^*m_{S^1}\). The final equality \[\int_{\X_1}S_v(\Omega_1)\ell_\xi(m_1)\Omega_1^n=\int_XS_v(\omega)\ell_\xi(\mu)\omega^n\] is another straightforward exercise in differentiating along a straight line path between equivariant representatives.
\end{proof}
		
\subsection{Weighted K-stability}
	
Here we review the definition of weighted K-(semi/poly)stability. This concept was initially introduced by Lahdili \cite{Lah19} in the case where the test configuration is smooth and K{\"a}hler. Later, Inoue \cite{Ino20} introduced weighted K-stability for arbitrary singular test configurations, and this is the theory we will use here. We assume the reader is familiar with equivariant cohomology and the de Rham representation of this theory, as well as equivariant locally finite homology and its duality with equivariant cohomology; see the appendix of \cite{Ino20} for all the necessary details.
		
Let \(Y\) be a \(T\)-equivariant complex manifold of dimension \(d\); later we will either take \(Y\) to be the K{\"a}hler manifold \(X\), an equivariant resolution \(\widetilde{\X}\) of a test configuration \(\X\) for \(X\), or the restriction \(\widetilde{\X}|_{\mathbb{C}}\) of such a resolution. In particular we do not require \(Y\) to be compact, for which reason we use locally finite homology in what follows. The torus \(T\) will either be \(T\) from previous sections, or we may later take \(T\times S^1\) for the action on a test configuration.
		
The map \(\pi:Y\to\mathrm{pt}\) is \(T\)-equivariant, so induces a pushforward map on equivariant locally finite homology: \[\pi_*:H_*^{\mathrm{lf},T}(Y;\mathbb{R})\to H_*^{\mathrm{lf},T}(\mathrm{pt};\mathbb{R}).\] Since \(Y\) and \(\mathrm{pt}\) are smooth, there are Poincar{\'e} duality isomorphisms in the equivariant theory: \[H^{\mathrm{lf},T}_{2k}(Y;\mathbb{R})\cong H_T^{2d-2k}(Y;\mathbb{R}),\] \[H^{\mathrm{lf},T}_{2k}(\mathrm{pt};\mathbb{R})\cong H_T^{-2k}(\mathrm{pt};\mathbb{R}).\] Using these isomorphisms, the map \(\pi_*\) on equivariant homology induces an integration map on equivariant cohomology, which we denote by the same symbol: \[\pi_*:H^{2k}_T(Y;\mathbb{R})\to H^{2k-2d}_T(\mathrm{pt};\mathbb{R}).\] The term ``integration map" is justified; indeed, if \(Y\) is compact and we take an equivariant form \(u\) on \(Y\) representing a class \([u]\in H^{2k}_T(Y;\mathbb{R})\), then \(\int_Y u\) represents \(\pi_*[u]\in H^{2k-2d}_T(\mathrm{pt};\mathbb{R})\). It will be convenient to reserve the notation \(\pi\) for test configurations, so we will instead denote \(\pi_*(\alpha)\) by just \((\alpha)\). Unlike in the non-equivariant case, we can have \((\alpha)\neq0\) even when \(\deg\alpha> 2d\). 
		
The following special case of the projection formula is derived from the observation that to integrate a smooth differential form, it suffices to integrate the form over a dense open subset. This will later be applied to show that the weighted Donaldson--Futaki invariant is well-defined.
		
\begin{lemma}[Projection formula]\label{lem:projection_formula}
	Let \(p:\widetilde{Y}\to Y\) be a proper \(T\)-equivariant morphism of \(T\)-equivariant complex manifolds, and assume there are \(T\)-invariant dense open subsets \(U\subset\widetilde{Y}\) and \(V\subset Y\) such that \(p(U)\subset V\) and \(p:U\to V\) is a biholomorphism. Let \(\beta\in H^{2k}_T(Y;\mathbb{R})\). Then \[(\beta)=(p^*\beta),\] where the former is an integral over \(Y\) and the latter an integral over \(\widetilde{Y}\).
\end{lemma}
		
Let \(h(x) = \sum_{k=0}^\infty\frac{a_k}{k!}x^k\) be a formal power series with coefficients \(a_k\in\mathbb{R}\). Given \(\beta\in H^2_T(Y;\mathbb{R})\), we define \[h(\beta):=\sum_{k=0}^\infty\frac{a_k}{k!}\beta^k\in \hat{H}^{*,\,\mathrm{even}}_T(Y,\mathbb{R})\] in the completion of the even-degree equivariant cohomology. The integration map \(\pi_*\) on cohomology then induces \[\pi_*:\hat{H}_T^{*,\,\mathrm{even}}(Y;\mathbb{R})\to\hat{H}_T^{*-2d,\,\mathrm{even}}(\mathrm{pt};\mathbb{R}).\] Given an element \(\zeta\in \hat{H}_T^{*,\,\mathrm{even}}(Y;\mathbb{R})\), we similarly write \((\zeta)\in \hat{H}_T^{*,\,\mathrm{even}}(\mathrm{pt};\mathbb{R})\) in place of \(\pi_*(\zeta)\). In particular, we have \((h(\beta))\in \hat{H}_T^{*,\,\mathrm{even}}(\mathrm{pt};\mathbb{R})\), and if \(\rho\in H^2_T(Y;\mathbb{R})\) is an auxiliary class, we also have \((\rho\,h(\beta))\in \hat{H}_T^{*,\,\mathrm{even}}(\mathrm{pt};\mathbb{R})\). The Chern--Weil theorem furnishes an isomorphism \[\hat{H}^{*,\mathrm{even}}_{T}(\mathrm{pt};\mathbb{R})\cong\mathbb{R}\llbracket\mathfrak{t}\rrbracket,\] where \(\mathfrak{t}\) is the Lie algebra of \(\mathfrak{t}\) and homogeneous polynomials on \(\mathfrak{t}\) of degree \(k\) are assigned degree \(2k\) in the ring. It follows we can identify \((\zeta)\) with a formal power series on the Lie algebra \(\mathfrak{t}\); we shall do this from now on. 

Let \((X,\alpha_T)\) be a \(T\)-equivariant compact K{\"a}hler manifold, and let \((\X,\mathcal{A}_T)\) be a test configuration for \((X,\alpha_T)\). We take \((\widetilde{\X},\widetilde{\A}_T)\) a \(T^{\mathbb{C}}\times \mathbb{C}^*\)-equivariant resolution of singularities of \((\X,\A_T)\). Recall that we have a moment map \(m\) for the \(T\times S^1\)-action on \(\X\) that is compatible with \(\mu\), in the sense that \(m(\X_\tau)=P:=\mu(X)\) for all \(\tau\in\mathbb{P}^1\), and \(\widetilde{\A}_T:=[\widetilde{\Omega}+\widetilde{m}_T]\). Define \(c_1(\widetilde{\X}/\mathbb{P}^1)_T:=c_1(\widetilde{\X})_T-c_1(\mathbb{P}^1)_T\), where the class \(c_1(\mathbb{P}^1)_T\) is tacitly pulled back from \(\mathbb{P}^1\) to \(\widetilde{\X}\). We note that the first Chern classes here arise from line bundles (namely the anticanonical bundles) which have canonical \(T\)-actions, so are automatically extended to equivariant cohomology.

Given all this data, we can define the equivariant intersections \[(c_1(\widetilde{\X}/\mathbb{P}^1)_Th(\widetilde{\A}_T)),\: (h(\widetilde{\A}_T)),\: (c_1(X)_Th(\alpha_T)),\:(h(\alpha_T)),\] all of which take values in \(\mathbb{R}\llbracket\mathfrak{t}\rrbracket\). It is a natural question to ask whether these power series converge. The following result of Inoue addresses this, and is the essential ingredient in defining the weighted Donaldson--Futaki invariant of a \(T\)-equivariant K{\"a}hler test configuration.
		
\begin{proposition}[{\cite[Proposition 3.8]{Ino20}}]\label{prop:convergence}
	Let \(\xi\in\mathfrak{t}\), and let \(h\) be a real analytic function on the subset \(\langle P,\xi\rangle\subset\mathbb{R}\), where \(P\subset\mathfrak{t}^*\) is the moment polytope of \((X,\alpha_T)\). As in Assumption \ref{ass:weights}, we assume the interval \(\langle P,\xi\rangle\) has \(0\) as its midpoint, so \(h\) has a convergent power series expansion about \(0\). Then the power series \[(c_1(\widetilde{\X}/\mathbb{P}^1)_Th(\widetilde{\A}_T)),\: (h(\widetilde{\A}_T)),\: (c_1(X)_Th(\alpha_T)),\:(h(\alpha_T))\in\mathbb{R}\llbracket\mathfrak{t}\rrbracket\] each converge absolutely on a neighbourhood of \(\xi\in\mathfrak{t}\). Furthermore, the first two are independent of the choice of equivariant resolution of singularities. Lastly, if \(h^{(n)}>0\) (resp. \(h^{(n+1)}>0\)) then \((h(\alpha_T))(\xi)>0\) (resp. \((h(\A_T))(\xi)>0\)).
\end{proposition}
		
In \cite[Proposition 3.8]{Ino20}, the function \(h\) is assumed to be entire on \(\mathbb{R}\), however the proof carries over exactly the same in the case described here. To see this, we will briefly describe how the proof works and how to adapt it to this setting.
		
\begin{proof}
	Write \(h(x) = \sum_{k=0}^\infty\frac{a_k}{k!}x^k\). For \(\eta\in\mathfrak{t}\) and \(N\geq n+1\), \begin{align*}
		\sum_{k=0}^N \frac{a_k}{k!} (\widetilde{\A}_T^k) (\eta)
		&=\sum_{k=0}^N \frac{a_k}{k!} \int_{\widetilde{\X}} (\widetilde{\Omega} + \langle \widetilde{m}_T, \eta \rangle)^k \\
		&= \int_{\widetilde{\X}} \sum_{k=0}^N \frac{a_k}{k!} (\widetilde{\Omega} + \langle \widetilde{m}_T, \eta \rangle)^k \\
		&= \frac{1}{(n+1)!} \int_{\widetilde{\X}} \left( \sum_{k=0}^{N-n-1} \frac{a_{k+n+1}}{k!} \langle \widetilde{m}_T, \eta \rangle^k \right) \widetilde{\Omega}^{n+1}.
	\end{align*} Since \(h\) is absolutely convergent on a neighbourhood of \(\langle P,\xi\rangle=\langle\widetilde{m}_T(\widetilde{\X}),\xi\rangle\), the sum \(\sum_{k=0}^\infty\frac{a_{k+n+1}}{k!}\langle\widetilde{m}_T,\eta\rangle^k\) converges absolutely and uniformly for all \(\eta\) in a neighbourhood of \(\xi\). Therefore, sending \(N\to\infty\) we may pass the limit through the integral to get \[\sum_{k=0}^\infty\frac{a_k}{k!}(\widetilde{\A}_T^k)(\eta)=\frac{1}{n!}\int_{\widetilde{\X}}h^{(n+1)}(\langle \widetilde{m},\eta\rangle)\widetilde{\Omega}^{n+1},\] and the convergence of the sum is absolute. From this formula, we also see the positivity of \((h(\widetilde{\A}_T))(\eta)\), given positivity of \(h^{(n+1)}\). Independence of the choice of resolution is straightforward using Lemma \ref{lem:projection_formula} as well as the fact that given two resolutions \(\X_1\to\X\) and \(\X_2\to\X\), there is a third resolution \(\X_3\to\X\) dominating \(\X_1\) and \(\X_2\).
		
	For the intersection \((c_1(\widetilde{\X}/\mathbb{P}^1)_Th(\widetilde{\A}_T))\), choose a Cartan representative \(\Sigma+\sigma\) for \(c_1(\widetilde{\X}/\mathbb{P}^1)_T\). We similarly compute \begin{align*}
		\sum_{k=0}^N\frac{a_k}{k!}(c_1(\widetilde{\X}/\mathbb{P}^1)_T\widetilde{\A}_T^k)(\eta)
		=&\sum_{k=0}^N\frac{a_k}{k!}\int_{\widetilde{\X}}(\Sigma+\langle\sigma,\eta\rangle)\wedge(\widetilde{\Omega}+\langle\widetilde{m}_T,\eta\rangle)^k \\
		=& \int_{\widetilde{\X}}\sum_{k=0}^N\frac{a_k}{k!}(\Sigma+\langle\sigma,\eta\rangle)\wedge(\widetilde{\Omega}+\langle\widetilde{m}_T,\eta\rangle)^k \\
		=& \frac{1}{n!}\int_{\widetilde{\X}}\sum_{k=0}^{N-n}\frac{a_{k+n}}{k!}\langle\widetilde{m}_T,\eta\rangle^k\Sigma\wedge\widetilde{\Omega}^n \\
		&+\frac{1}{(n+1)!}\int_{\widetilde{\X}}\sum_{k=0}^{N-n-1}\frac{a_{k+n+1}}{k!}\langle\sigma,\eta\rangle\langle\widetilde{m}_T,\eta\rangle^k\widetilde{\Omega}^{n+1}.
	\end{align*} Sending \(N\to\infty\), we have uniform convergence in the integrals and \begin{align*}
		\sum_{k=0}^\infty\frac{a_k}{k!}(c_1(\widetilde{\X}/\mathbb{P}^1)_T\widetilde{\A}_T^k)(\eta)=&\frac{1}{n!}\int_{\widetilde{\X}}h^{(n)}(\langle\widetilde{m}_T,\eta\rangle)\Sigma\wedge\widetilde{\Omega}^n \\ &+\frac{1}{(n+1)!}\int_{\widetilde{\X}}\langle\sigma,\eta\rangle h^{(n+1)}(\langle\widetilde{m}_T,\eta\rangle)\widetilde{\Omega}^{n+1}.
	\end{align*} Hence the intersection \((c_1(\widetilde{\X}/\mathbb{P}^1)_Tf(\widetilde{\A}_T))\) converges absolutely on a neighbourhood of \(\xi\). It remains to see this is independent of the choice of resolution. It suffices to show that if we have a tower \(\widetilde{\X}'\to\widetilde{\X}\to\X\) of equivariant resolutions, then the intersection numbers on \(\widetilde{\X}'\) and \(\widetilde{\X}\) agree. This is straightforward however; the difference \(c_1(\widetilde{\X}'/\mathbb{P}^1)_T-c_1(\widetilde{\X}/\mathbb{P}^1)_T\) is a sum of equivariant divisors on \(\widetilde{\X}'\), which all map to analytic subvarieties of \(\X\) with dimension at most \(n-2\). It follows from the projection formula that the intersection of \((\widetilde{\A}_T')^{n-1}\) with these divisors vanishes, so the divisors do not contribute to the equivariant intersections.
\end{proof}
		
\begin{definition}[{\cite[Proposition 3.26]{Ino20}}]
	Let \((\X,\A_T)\) be a test configuration for \((X,\alpha_T)\), and choose a \(T^{\mathbb{C}}\times\mathbb{C}^*\)-equivariant resolution of singularities \(\widetilde{\X}\to\X\). The \emph{weighted Donaldson--Futaki invariant} of \((\X,\A_T)\) is  \[\DF_{v, w}(\X, \A_T) := \left[-(c_1(\widetilde{\X}/\mathbb{P}^1)_T f(\widetilde{\A}_T))+\frac{(c_1(X)_T f'(\alpha_T))}{(g'(\alpha_T))}(g(\widetilde{\A}_T))\right](\xi),\] and is independent of the choice of equivariant resolution.
\end{definition}
				
Here the expression between the square brackets is an analytic function on \(\mathfrak{t}\) defined in a neighbourhood of \(\xi\in\mathfrak{t}\), and the expression is evaluated at \(\xi\) to produce a real number that is the weighted Donaldson--Futaki invariant. We remark that \((g')^{(n)}=g^{(n+1)}>0\), so \((g'(\alpha_T))(\xi)>0\) and hence \((g'(\alpha_T))\) is non-zero on a neighbourhood of \(\xi\).

\begin{remark}
	Lahdili has also introduced a weighted Donaldson--Futaki invariant in the case the total space \(\X\) of the K{\"a}hler test configuration is smooth \cite[Definition 11]{Lah19}. By \cite[Proposition 3.26 (3)]{Ino20}, this agrees with the invariant \(\DF_{v,w}(\X,\A_T)\) defined above, and hence Inoue's invariant provides an extension of Lahdili's invariant to singular test configurations.
\end{remark}

From the proof of \cite[Proposition 3.6]{Ino20} we have:

\begin{lemma}[\cite{Ino20}]\label{lem:average_weighted_scalar_curvature}
	\[\frac{(c_1(X)_T f'(\alpha_T))(\xi)}{(g'(\alpha_T))(\xi)}=\hat{S}_{v,w},\] where \(\hat{S}_{v,w}\) is the average weighted scalar curvature defined in equation \eqref{eq:average_weighted_scalar_curvature}.
\end{lemma}
%
				
We will also need some more general results concerning the equivariant calculus. In particular, instead of pushing forward equivariant classes to a point, it will be useful to push them to \(\mathbb{P}^1\) or \(\mathbb{C}\), to apply localisation in equivariant cohomology. 

Let \((\X,\A_T)\) be a test configuration for \((X,\alpha_T)\), with a choice of equivariant resolution \(\widetilde{\X}\to\X\). We denote by \(\widetilde{\X}_{\mathbb{C}}\) the restriction of \(\widetilde{\X}\) to \(\mathbb{C}\), and write \(\pi:\widetilde{\X}_{\mathbb{C}}\to\mathbb{C}\) for the projection. This induces an integration over the fibres map \[\pi_*:H^{2k}_{T\times S^1}(\widetilde{\X}_{\mathbb{C}};\mathbb{R})\to H^{2k-2n}_{T\times S^1}(\mathbb{C};\mathbb{R}),\] where the \(T\)-action on \(\mathbb{C}\) is taken to be trivial. Given an equivariant class \(\rho\in H^{2k}_{T\times S^1}(\widetilde{\X}_{\mathbb{C}};\mathbb{R})\) we write \((\rho)_{\mathbb{C}}\in H^{2k-2n}_{T\times S^1}(\mathbb{C};\mathbb{R})\) for the image of \(\rho\) under \(\pi_*\), and extend this to the completion of the equivariant cohomology rings of \(\widetilde{\X}\) and \(\mathbb{C}\). In particular, we have \[(c_1(\widetilde{\X}/\mathbb{P}^1)_{T\times S^1}f(\widetilde{\A}_{T\times S^1}))_{\mathbb{C}},\:(g(\widetilde{\A}_{T\times S^1}))_{\mathbb{C}}\in \hat{H}^*_{T\times S^1}(\mathbb{C};\mathbb{R}).\] Notice we use \(T\times S^1\)-equivariant cohomology since we are working with the non-compact test configuration \(\X_{\mathbb{C}}\), whereas previously the information in the \(S^1\)-action was already encoded in the compactification.

Now, since the \(T\)-action on \(\mathbb{C}\) is trivial, there is a Kunneth decomposition \[H^{2\ell}_{T\times S^1}(\mathbb{C};\mathbb{R})=\bigoplus_{j+k=\ell}H^{2j}_T(\mathrm{pt};\mathbb{R})\otimes H^{2k}_{S^1}(\mathbb{C};\mathbb{R}).\] By the Chern--Weil isomorphism, \(H^{2j}_T(\mathrm{pt};\mathbb{R})\) is the space \(S^{j}\mathfrak{t}^*\) of homogeneous degree \(j\) polynomials on \(\mathfrak{t}\). Thus, given an element \(\xi\in\mathfrak{t}\), there is an evaluation map \[\mathrm{el}_\xi:H^*_{T\times S^1}(\mathbb{C};\mathbb{R})\to H^*_{S^1}(\mathbb{C};\mathbb{R}),\] given by evaluating the \(H^*_T(\mathrm{pt};\mathbb{R})\)-component at \(\xi\). Following \cite{Ino20}, define \[D_\xi:H^*_{T\times S^1}(\mathbb{C};\mathbb{R})\to H^2_{S^1}(\mathbb{C};\mathbb{R})\] to be the composition of \(\mathrm{el}_\xi\) with projection to the degree 2 component in equivariant cohomology. Note there is an \(S^1\)-equivariant deformation retraction \(\mathbb{C}\to\mathrm{pt}\), so \[H^2_{S^1}(\mathbb{C};\mathbb{R})\cong H^2_{S^1}(\mathrm{pt};\mathbb{R})=H^2(\mathbb{CP}^\infty;\mathbb{R})=\mathbb{R}\cdot c_1(\mathcal{O}(1)),\] where \(c_1(\mathcal{O}(1))\) is the first chern class of the universal line bundle \(\mathcal{O}(1)\) on \(\mathbb{CP}^\infty\). Thus, taking the coefficient of \(c_1(\mathcal{O}(1))\), we may instead consider \(D_\xi\) as a map \[D_\xi:H^*_{T\times S^1}(\mathbb{C};\mathbb{R})\to\mathbb{R}.\] The following result of Inoue is a consequence of the localisation formula in equivariant cohomology:

\begin{proposition}[{\cite[Proposition 3.19]{Ino20}}]\label{prop:localisation}
	Let \((\X,\mathcal{A}_T)\) be a test configuration for \((X,\alpha_T)\). Let \(\xi\in\mathfrak{t}\), and suppose that \(f\) is a power series with real coefficients that converges absolutely on the subset \(\langle P,\xi\rangle\subset\mathbb{R}\), where \(P\subset\mathfrak{t}^*\) is the moment polytope of \((X,\alpha_T)\). Choose a \(T^{\mathbb{C}}\times \mathbb{C}^*\)-equivariant resolution of singularities \(\widetilde{\X}\to\X\). Then \[D_\xi(f(\widetilde{\A}_{T\times S^1}))_{\mathbb{C}} = (f(\widetilde{\A}_T))(\xi),\] and \[D_\xi(c_1(\widetilde{\X}_{\mathbb{C}}/\mathbb{C})_{T\times S^1}f(\widetilde{\A}_{T\times S^1}))_{\mathbb{C}}=(c_1(\widetilde{\X}/\mathbb{P}^1)_T f(\widetilde{\A}_T))(\xi).\]
\end{proposition}

\begin{remark}
	In order to compute the intersections over \(\mathbb{C}\) we can proceed exactly as in the proof of Proposition \ref{prop:convergence}. That is, we take a resolution of singularities \(p:\widetilde{\X}\to\X\), pull back \(\Omega\) and its moment map \(m:\X\to\mathfrak{t}^*\oplus\mathbb{R}\) for the \(T\times S^1\)-action to get an equivariant representative \(\widetilde{\Omega}+\widetilde{m}\) for \(\widetilde{\A}_{T\times S^1}:=p^*\A_{T\times S^1}\). For the class \(c_1(\widetilde{\X}/\mathbb{P}^1)_{T\times S^1}\), we simply choose an equivariant representative \(\Sigma+\sigma\). Then we can integrate equivariant forms over the fibre to produce an equivariant representative on the base \(\mathbb{C}\). We comment that this is only a \emph{continuous} representative in general; since \(\widetilde{\X}\to\mathbb{C}\) is not a submersion, the fibre integrals will not necessarily produce smooth forms on the base, hence we only have an equivariant \emph{current} on the base representing the equivariant intersection.
\end{remark}
	
\section{Relative weighted K-stability}

In this section we study relative weighted K-stability. One such definition was previously given by Lahdili \cite{Lah20b}, who related the condition to existence of a weighted extremal metric. Here we give a separate treatment using inner products associated to test configurations. After showing the two definitions are equivalent, we use Lahdili's work to argue that a weighted extremal manifold is relatively weighted K-semistable for our definition. This will be upgraded to relative weighted K-polystability in the next section.

\subsection{Twisting test configurations}
		
Since we are dealing with non-projective K{\"a}hler manifolds, we will need a mild generalisation of test configurations that allows one to incorporate irrational elements of the Lie algebra \(\mathfrak{t}\); such test configurations were used in \cite{Der18}. We first describe how integral elements of \(\mathfrak{t}\) allow one to twist test configurations, and then we will extend the definition to arbitrary elements of \(\mathfrak{t}\). 

Let \(\beta:\mathbb{C}^*\to T^{\mathbb{C}}\) be a 1-parameter subgroup. Such a \(\beta\) determines a \emph{product test configuration} \((\X_\beta,\A_T)\), first by letting \[(\X_\beta,\A_T)|_{\mathbb{C}}:=(X\times\mathbb{C},p^*\alpha_T)\] where the \(\mathbb{C}^*\)-action on the right-hand side is \(s(x,t):=(\beta(s)x,st)\) for \(s\in\mathbb{C}^*\), and then compactifying trivially over \(\infty\in\mathbb{P}^1\), in the sense that the restriction to \(\mathbb{P}^1\backslash\{0\}\) is trivial. Letting \(h_\beta:X\to\mathbb{R}\) be a hamiltonian function for the generator of \(\beta\), the Donaldson--Futaki invariant \(\DF_{v,w}(\X_\beta,\A_T)\) of the product test configuration is equal to \begin{equation}\label{eq:weighted_Futaki}
	F_{v,w}(\beta):=\frac{1}{n!}\int_X(\hat{S}_{v,w}-S_{v,w}(\omega))h_\beta w(\mu) \omega^n,
\end{equation} by \cite[Proposition 3]{Lah19}. Identifying \(\beta\) with an integral element of the Lie algebra \(\mathfrak{t}\) of \(T\), note this formula extends naturally to all of \(\mathfrak{t}\), even to irrational elements which do not define test configurations. It may be that such an irrational element destabilises the manifold \((X,\omega)\) (see for example \cite{ACGTF08}) so we will incorporate these into our definition of stability.

\begin{definition}
	Let \((\X,\mathcal{A}_T)\) be a \(T\)-equivariant K{\"a}hler test configuration for \((X,\alpha_T)\), and let \(\beta\in\mathfrak{t}\). Then we define \[\DF_{v,w}(\X,\A_T,\beta):=\DF_{v,w}(\X,\A_T)+F_{v,w}(\beta),\] where \(F_{v,w}(\beta)\) is defined in \eqref{eq:weighted_Futaki}.
\end{definition}

Given a test configuration \((\X,\A_T)\) for \((X,\alpha_T)\) and a 1-parameter subgroup \(\beta:\mathbb{C}^*\to T^{\mathbb{C}}\), we can \emph{twist} the test configuration by the 1-parameter subgroup to produce a new test configuration. To do this, we restrict \((\X,\mathcal{A}_T)\) to \(\mathbb{C}\), then define a new \(\mathbb{C}^*\) action to be the original \(\mathbb{C}^*\)-action on \(\X\) multiplied by the 1-parameter subgroup \(\beta\). Compactifying trivially over infinity, we then produce a new test configuration \((\X',\A'_T)\) for \((X,\alpha_T)\). We claim: \begin{lemma}\label{lem:integral_twist} 
	Let \((\X',\A'_T)\) denote the twist of a test configuration \((\X,\A_T)\) by a one-parameter subgroup \(\beta:\mathbb{C}^*\to T^{\mathbb{C}}\). Then
	\[\DF_{v,w}(\X',\A_T')=\DF_{v,w}(\X,\A_T)+F_{v,w}(\beta).\]
\end{lemma} 

\begin{proof}
	By Proposition \ref{prop:localisation}, we have \[(g(\widetilde{\A}_T'))(\xi)=D_\xi(g(\widetilde{\A}'_{T\times S^1_\beta}))_{\mathbb{C}},\] where we write \(S^1_\beta\) to remind ourselves the circle action is twisted by \(\beta\). The class \(\widetilde{\A}'_{T\times S^1_\beta}|_{\mathbb{C}}\) has as an equivariant representative \[\widetilde{\Omega}+\widetilde{m}_T+\widetilde{m}_{S^1_\beta}=\widetilde{\Omega}+\widetilde{m}_T+(\widetilde{m}_{S^1}+\langle \widetilde{m}_T,\beta\rangle),\] where \(\widetilde{m}_T+(\widetilde{m}_{S^1}+\langle \widetilde{m}_T,\beta\rangle)\) takes values in \(\mathfrak{t}^*\oplus \mathbb{R}=\mathrm{Lie}(T\times S^1)^*\). For the remainder of the proof we shall omit the excessive tilde notation, although the reader should keep in mind objects are always integrated over the resolution \(\widetilde{\X}\). Evaluating at \(\xi\in\mathfrak{t}\), an equivariant current representative for \(D_\xi(g(\mathcal{A}'_{T\times S^1_\beta}))_{\mathbb{C}}\) is \begin{align}\label{eq:dummy}
		&\left[\int_{\X'_{\mathbb{C}} / \mathbb{C}} \sum_{k=0}^\infty \frac{b_k}{k!} (\Omega + m_{S^1} + \langle m_T, \beta \rangle + \langle m_T, \xi \rangle)^k\right]_{\mathrm{deg=2}} \nonumber \\
		=&\int_{\X'_{\mathbb{C}} / \mathbb{C}} \sum_{k=n+1}^\infty \frac{b_k}{k!} \binom{k}{n+1} (\Omega + m_{S^1} + \langle m_T, \beta \rangle)^{n+1} \langle m_T, \xi \rangle^{k-(n+1)} \nonumber \\
		=&\frac{1}{(n+1)!} \int_{\X'_{\mathbb{C}} / \mathbb{C}} \sum_{k=0}^\infty \frac{b_{k+n+1}}{k!} (\Omega + m_{S^1} + \langle m_T, \beta \rangle)^{n+1} \langle m_T, \xi \rangle^k \nonumber \\
		=&\frac{1}{(n+1)!} \int_{\X'_{\mathbb{C}} / \mathbb{C}} g^{(n+1)}(\langle m_T, \xi \rangle) (\Omega + m_{S^1} + \langle m_T,\beta \rangle)^{n+1} \nonumber \\
		=& \frac{1}{(n+1)!} \int_{\X'_{\mathbb{C}} / \mathbb{C}} w(m_T) (\Omega + m_{S^1})^{n+1} + \frac{1}{n!} \int_{\X'_{\mathbb{C}} / \mathbb{C}} w(m_T) \langle m_T, \beta \rangle \Omega^n.
	\end{align} A similar calculation (or setting \(\beta=0\) above) shows that the first term of \eqref{eq:dummy} is an equivariant representative for \[D_\xi(g(\A_{T\times S^1}))_{\mathbb{C}}=(g(\mathcal{A}_T))(\xi),\] i.e. the corresponding invariant for the untwisted test configuration (note the restrictions \(\X'_{\mathbb{C}}\) and \(\X_{\mathbb{C}}\) are identical, by definition of \(\X'\)). By Lemma \ref{lem:independence}, the second term of \eqref{eq:dummy} is a constant function on the base, equal to \(\frac{1}{n!}\int_X h_\beta w(\mu)\omega^n\). We therefore have \[(g(\A'_T))(\xi)=(g(\A_T))(\xi)+\frac{1}{n!}\int_X h_\beta w(\mu)\omega^n.\]
	
	We next calculate \[(c_1(\X'/\mathbb{P}^1)_Tf(\mathcal{A}'_T))(\xi)=D_\xi(c_1(\X'_{\mathbb{C}}/\mathbb{C})_{T\times S^1_\beta}f(\A'_{T\times S^1_\beta}))_{\mathbb{C}}.\] Choosing an equivariant representative \[\Sigma+\sigma_T+\sigma_{S^1_\beta}=\Sigma+\sigma_T+(\sigma_{S^1}+\langle\sigma_T,\beta\rangle)\] of \(c_1(\X'_{\mathbb{C}}/\mathbb{C})_{T\times S^1_\beta}\) on \(\X'_{\mathbb{C}}\), we get \begingroup \allowdisplaybreaks \begin{align*}
		&\left[\int_{\X'_{\mathbb{C}} / \mathbb{C}} (\Sigma + \sigma_{S^1_\beta} + \langle \sigma_T, \xi \rangle) \sum_{k=0}^\infty \frac{a_k}{k!} (\Omega + m_{S^1_\beta} + \langle m_T, \xi \rangle)^k\right]_{\mathrm{deg}=2} \\
		=& \int_{\X'_{\mathbb{C}} / \mathbb{C}} \langle \sigma_T, \xi \rangle \sum_{k=n+1}^\infty \frac{a_k}{k!} \binom{k}{n+1} (\Omega+m_{S^1} + \langle m_T, \beta \rangle)^{n+1} \langle m_T, \xi \rangle^{k-(n+1)}  \\
		&+\int_{\X'_{\mathbb{C}} / \mathbb{C}} (\Sigma + \sigma_{S^1} + \langle \sigma_T, \beta \rangle) \sum_{k=n}^\infty \frac{a_k}{k!} \binom{k}{n} (\Omega + m_{S^1} + \langle m_T, \beta \rangle)^n \langle m_T, \xi \rangle^{k-n} \\
		=&\frac{1}{(n+1)!} \int_{\X'_{\mathbb{C}} / \mathbb{C}} \langle \sigma_T, \xi \rangle \sum_{k=0}^\infty \frac{a_{k+n+1}}{k!} (\Omega + m_{S^1} + \langle m_T, \beta \rangle)^{n+1} \langle m_T, \xi \rangle^{k}  \\
		&+\frac{1}{n!} \int_{\X'_{\mathbb{C}} / \mathbb{C}} (\Sigma + \sigma_{S^1} + \langle \sigma_T, \beta \rangle) \sum_{k=0}^\infty \frac{a_{k+n}}{k!} (\Omega + m_{S^1} + \langle m_T, \beta \rangle)^n \langle m_T, \xi \rangle^k \\
		=& \frac{1}{(n+1)!} \int_{\X'_{\mathbb{C}} / \mathbb{C}} \langle \sigma_T, \xi \rangle f^{(n+1)}(\langle m_T, \xi \rangle) (\Omega + m_{S^1} + \langle m_T, \beta \rangle)^{n+1} \\
		&+\frac{1}{n!}\int_{\X'_{\mathbb{C}} / \mathbb{C}}f^{(n)}(\langle m_T, \xi \rangle)(\Sigma + \sigma_{S^1} + \langle \sigma_T, \beta \rangle)(\Omega + m_{S^1} + \langle m_T, \beta \rangle)^n.
	\end{align*}\endgroup Expanding out these final lines, we have an equivariant representative of \[D_\xi(c_1(\X_{\mathbb{C}}/\mathbb{C})_{T\times S^1}f(\A_{T\times S^1}))_{\mathbb{C}}=(c_1(\X/\mathbb{P}^1)_Tf(\mathcal{A}_T))(\xi)\] for the untwisted test configuration, plus the remaining terms involving \(\beta\): \begin{align}\label{eq:remaining_terms}
		&\frac{1}{n!} \int_{\X_{\mathbb{C}}/\mathbb{C}}\langle m_T,\beta\rangle\langle\sigma_T,\xi\rangle \widetilde{v}'(\langle m_T,\xi\rangle)\Omega^n \\
		+&\frac{1}{n!} \int_{\X_{\mathbb{C}}/\mathbb{C}}\langle\sigma_T,\beta\rangle\widetilde{v}(\langle m_T,\xi\rangle)\Omega^n \nonumber \\
		+&\frac{1}{(n-1)!} \int_{\X_{\mathbb{C}}/\mathbb{C}}\langle m_T,\beta\rangle \widetilde{v}(\langle m_T,\xi\rangle)\Sigma\wedge\Omega^{n-1}. \nonumber
	\end{align} We claim that these remaining terms are (fibrewise) independent of the choice of representative \(\Sigma+\sigma_T\), and are in fact constant. To see the independence, suppose we have another representative \(\Sigma'+\sigma_T'\). We may write \(\Sigma' = \Sigma+i\d\db\psi\), in which case \(\sigma'_T = \sigma_T+d^c\psi\). Differentiating along the straight line \(t(\Sigma'+\sigma_T')+(1-t)(\Sigma+\sigma_T)\) joining these representatives, we get \begin{align*}
		&\frac{1}{n!} \int_{\X_{\mathbb{C}}/\mathbb{C}}\langle m_T,\beta\rangle d^c\psi(\xi) \widetilde{v}'(\langle m_T,\xi\rangle)\Omega^n \\
		+&\frac{1}{n!} \int_{\X_{\mathbb{C}}/\mathbb{C}}d^c\psi(\beta)\widetilde{v}(\langle m_T,\xi\rangle)\Omega^n \\
		+&\frac{1}{n!} \int_{\X_{\mathbb{C}}/\mathbb{C}}\langle m_T,\beta\rangle \widetilde{v}(\langle m_T,\xi\rangle)n\,i\d\db\psi\wedge\Omega^{n-1}.
	\end{align*} Integrating by parts and using the moment map property, the final term cancels with the first two terms, hence we have independence.
	
	To see that the terms \eqref{eq:remaining_terms} are constant, we may choose on any given fibre \(\X_\tau\) for \(\tau\neq0\) the representative \(\Ric(\Omega_\tau)+\Delta m_T\), where \(\Delta\) is the Laplacian of \(\Omega_\tau\). The terms \eqref{eq:remaining_terms} then reduce to \begin{align*}
		&\frac{1}{n!} \int_{\X_\tau}\langle m_T,\beta\rangle((\Delta m_T^\xi) \widetilde{v}'(m_T^\xi) + \Delta(\widetilde{v}(m_T^\xi))+\widetilde{v}(m_T^\xi)S(\Omega_\tau))\Omega_\tau^n \\
		=&\frac{1}{n!}\int_{\X_\tau}\langle m_T,\beta\rangle S_v(\Omega_\tau)\Omega_\tau^n.
	\end{align*} By Lemma \ref{lem:independence}, this is a constant function on \(\mathbb{C}^*\). By continuity, the terms \eqref{eq:remaining_terms} are therefore constant on \(\mathbb{C}\), and equal to \[\frac{1}{n!}\int_X h_\beta S_v(\omega)\omega^n=\frac{1}{n!}\int_X  S_{v,w}(\omega)h_\beta w(\mu)\omega^n.\]
	
	We now substitute these equalities into the formula for the weighted Donaldson--Futaki invariant: \begin{align*}
		\DF_{v, w}(\X', \A'_T) :=& \left[-(c_1(\X' / \mathbb{P}^1)_T f(\A'_T))+\frac{(c_1(X)_T f'(\alpha_T))}{(g'(\alpha_T))}(g(\A'_T))\right](\xi) \\
		=&-(c_1(\X/\mathbb{P}^1)_T f(\mathcal{A}_T))(\xi)-\frac{1}{n!}\int_X S_{v,w}(\omega)h_\beta w(\mu)\omega^n \\
		&+\frac{(c_1(X)_T f'(\alpha_T))}{(g'(\alpha_T))} \left(g(\A_T)(\xi)+\frac{1}{n!}\int_X h_\beta w(\mu)\omega^n\right) \\
		=& \DF_{v, w}(\X, \A_T) + \frac{1}{n!}\int_X(\hat{S}_{v,w}-S_{v,w}(\omega))h_\beta w(\mu) \omega^n\\
		=&\DF_{v, w}(\X, \A_T) + F_{v,w}(\beta).
	\end{align*} Here we have used Lemma \ref{lem:average_weighted_scalar_curvature}, and the proof is complete.
\end{proof}

\begin{remark}
	In \cite[Proposition 4.3]{Der18}, the corresponding result in the unweighted case was proved using asymptotics of the Mabuchi functional; the same proof works in this setting with the weighted Mabuchi functional.
\end{remark}

\begin{definition}[{\cite[Definition 1.8]{Ino20}}]
	A \(T\)-equivariant K{\"a}hler manifold \((X,\alpha_T)\) is: \begin{enumerate}
		\item \emph{weighted K-semistable} if \(\DF_{v,w}(\X,\A_T,\beta)\geq0\) for all \(T\)-equivariant normal K{\"a}hler test configurations \((\X,\A_T)\) for \((X,\alpha_T)\) and all elements \(\beta\in\mathfrak{t}\),
		\item \emph{weighted K-polystable} if it is weighted K-semistable, and \(\DF_{v,w}(\X,\A_T,\beta)=0\) only if \((\X,\A_T)\) is isomorphic to a \(T\)-equivariant product test configuration,
		\item \emph{weighted K-stable} if it is weighted K-polystable, and \(\mathrm{Aut}_0(X)^T=T^{\mathbb{C}}\), where \(\mathrm{Aut}_0(X)^T\) is the group of \(T\)-equivariant reduced automorphisms of \(X\).
	\end{enumerate}
\end{definition}

Note that if \(\beta\) is a rational element of \(\mathfrak{t}\) and \(k\in\mathbb{N}\) is such that \(k\beta\) is integral, then \begin{align*}
	\DF_{v,w}(\X,\A_T,\beta) :=&\DF_{v,w}(\X,\A_T)+F_{v,w}(\beta) \\
	=& \frac{1}{k}(k\DF_{v,w}(\X,\A_T)+kF_{v,w}(\beta)) \\
	\geq& \frac{1}{k}(\DF_{v,w}(\rho_k^*(\X,\A_T))+ F_{v,w}(k\beta)),
\end{align*} where \(\rho_k\) is the map \(\mathbb{P}^1\to\mathbb{P}^1\), \(z\mapsto z^k\), and we write \(\rho_k^*(\X,\A_T)\) for the normalised pulled back test configuration under \(\rho_k\); see \cite[Proposition 7.14]{BHJ19} for the final inequality in the unweighted case, the weighted case is proved entirely similarly by introducing a weighted analogue of the non-Archimedean Mabuchi functional: \[M^{\mathrm{NA}}_{v,w}(\X,\A_T) := \DF_{v,w}(\X,\A_T) + \frac{((\X_{0,\mathrm{red}}-\X_0)_T g'(\A_T))(\xi)}{(g'(\alpha_T))(\xi)}.\] By Lemma \ref{lem:integral_twist}, the bottom line \(\frac{1}{k}(\DF_{v,w}(\rho_k^*(\X,\A_T))+ F_{v,w}(k\beta))\) is \(\frac{1}{k}\) times the weighted Donaldson--Futaki invariant of the test configuration \(\rho_k^*(\X,\A_T)\) twisted by \(k\beta\). 

By approximating an irrational \(\beta\) by rational elements of \(\mathfrak{t}\), we can therefore deduce:

\begin{lemma}
	\((X,\alpha_T)\) is weighted K-semistable if and only if \(\DF_{v,w}(\X,\A_T)\geq0\) for all test configurations \((\X,\A_T)\) for \((X,\alpha_T)\). 
\end{lemma}

That is, to test weighted K-semistability one need not consider twists by elements of the Lie algebra of \(T\).
			
\subsection{Relative weighted K-(semi/poly)stability}

In this section we introduce notions of relative weighted K-(semi/poly)stability. In the non-weighted setting, relative K-stability was initially introduced by Sz{\'e}kelyhidi \cite{Sze07} for algebraic varieties, and later extended to the K{\"a}hler setting by Dervan \cite{Der18}. 

Although relative weighted K-semistability does not appear to be explicitly defined in Lahdili's work, it is actually encompassed by the definition of weighted K-semistability in \cite{Lah19,Lah20b}. The reason for the difference in terminology here is that in Lahdili's work, the weight function \(w\) is permitted to take non-positive values. Thus, in \cite{Lah19,Lah20b}, a weighted extremal metric is a special example of a weighted cscK metric. 

Here we will use different terminology, restricting the weight function \(w\) to be always positive, and explicitly separating out the cscK case from the extremal case. Furthermore, we will give a slightly different definition of relative weighted K-semistability to what is in \cite{Lah19,Lah20b}, more along the lines of \cite{Sze07, Der18}.

To define relative weighted K-stability, we first introduce weighted inner products determined by test configurations. Let \((\X,\A_T)\) be a test configuration, and let \(\zeta_1, \zeta_2\in\mathrm{Lie}(T\times S^1)\). Denote by \(h_1\) and \(h_2\) the hamiltonian functions on \(\X\) of \(\zeta_1\) and \(\zeta_2\) constructed in Lemma \ref{lem:hamiltonians}, and let \(\hat{h}_1\) and \(\hat{h}_2\) be their fibrewise averages with respect to the measure \(w(m_T|_{\X_\tau})\Omega_\tau^n\).

\begin{definition}
	The \emph{weighted inner product} on \(\mathrm{Lie}(T\times S^1) = \mathfrak{t}\oplus \mathbb{R}\) determined by \((\X,\A_T)\) is \[\langle\zeta_1,\zeta_2\rangle_{(\X,\A_T)}:=\int_{\X_0}(h_1-\hat{h}_1)(h_2-\hat{h}_2)w(m_0)\Omega_0^n.\]
\end{definition}

It is clear that this is a genuine inner product, i.e. is bilinear and positive definite.

\begin{lemma}\label{lem:inner_product}
	The inner product \(\langle-,-\rangle_{(\X,\A_T)}\) is independent of the choice of equivariant representative \(\Omega+m_T\) for \(\A_T\). Furthermore, when \(\zeta_1,\zeta_2\in\mathfrak{t}\), the inner product is independent of the test configuration itself, and may be computed on \((X,\alpha_T)\).
\end{lemma}

Hence, when we write the inner product of elements of \(\mathfrak{t}\), we shall omit the test configuration from the notation.

\begin{proof}
	First assume \(\zeta_1,\zeta_2\in\mathfrak{t}\). By Lemma \ref{lem:independence}, the integral of \((h_1-\hat{h}_1)(h_2-\hat{h}_2)w(m_\tau)\Omega_\tau^n\) over the general fibre \(\X_\tau\) of \(\X\) is independent of \(\tau\in\mathbb{C}^*\), and is equal to the integral of \((h'_1-\hat{h}'_1)(h'_2-\hat{h}'_2)w(\mu)\omega^n\) over \(X\), where \(h_j'\) is the hamiltonian for \(\zeta_j\) on \((X,\omega)\). Since \[\int_{\X_0}(h_1-\hat{h}_1)(h_2-\hat{h}_2)w(m_0)\Omega_0^n = \lim_{\tau\to0}\int_{\X_\tau}(h_1-\hat{h}_1)(h_2-\hat{h}_2)w(m_\tau)\Omega_\tau^n,\] we have \[\langle\zeta_1,\zeta_2\rangle_{(\X,\A_T)} = \int_X(h_1'-\hat{h}'_1)(h_2'-\hat{h}_2')w(\mu)\omega^n,\] and the inner product depends only on \((X,\alpha_T)\).
	
	In the case where \(\zeta_1,\zeta_2\) may also take values in \(\mathrm{Lie}(S^1)\), we can only compute the integral over \(\X_0\). In this case, the integral takes the form \[\int_{\X_0}h(m_{T\times S^1})\Omega_0^n,\] where \(h\) is a function on the moment polytope of \(T\times S^1\). Let \(\widetilde{\X}\to\X\) be a \(T^{\mathbb{C}}\times\mathbb{C}^*\)-equivariant resolution of singularities so that the central fibre \(\widetilde{\X}_0\) is a simple normal crossings divisor. The integral can then be computed as \[\int_{\widetilde{\X}_0}h(\widetilde{m}_{T\times S^1})\widetilde{\Omega}_0^n = \sum_{j = 1}^N a_j \int_{\mathcal{Y}_j}h(\widetilde{m}_{T\times S^1})\widetilde{\Omega}_0^n,\] where \(\mathcal{Y}_1,\ldots,\mathcal{Y}_N\) denote the reduced components of central fibre \(\widetilde{\X}_0\), which have multiplicities \(a_1,\ldots,a_N\) respectively. Each \(\mathcal{Y}_j\) is a smooth manifold, and the integral over \(\mathcal{Y}_j\) is equivariant cohomological so is independent of the choice of representative for the equivariant cohomology class \([\widetilde{\Omega}+\widetilde{m}]|_{\mathcal{Y}_j}\). Hence the original integral does not depend on the choice of representative \([\Omega+m]\) for \(\mathcal{A}_T\).
\end{proof}

We now introduce the weighted \(L^1\)-norm of a test configuration, although we will not need this until later.

\begin{definition}
	Let \((\X,\A_T)\) be a test configuration for \((X,\alpha_T)\), and let \(\lambda\) denote the \(\mathbb{C}^*\)-action of the test configuration, with hamiltonian function \(h_\lambda\). Then the \emph{weighted \(L^1\)-norm} of the test configuration is \[\|(\X,\A_T)\|_1^w := \int_{\X_0}|h_\lambda - \hat{h}_\lambda| w(m_0) \Omega_0^n.\]
\end{definition}

Similar to the inner product, one may prove the norm is independent of the choice of equivariant representative for \(\A_T\).

\begin{definition}
	Let \((X,\alpha_T)\) be a \(T\)-equivariant K{\"a}hler manifold, and let \((\X,\A_T)\) be a \(T\)-equivariant K{\"a}hler test configuration for \((X,\alpha_T)\). Denote by \(\lambda\) the generator of the \(\mathbb{C}^*\)-action of the test configuration, and let \(\{\beta_j\}_{j=1}^r\) be an orthonormal basis for \(\mathfrak{t}\). We define \[\DF^T_{v,w}(\X,\A_T):=\DF_{v,w}(\X,\A_T)-\sum_{j=1}^r\langle\lambda,\beta_j\rangle_{(\X,\A_T)}F_{v,w}(\beta_j).\]
\end{definition} 

This is easily observed to be independent of the choice of orthonormal basis for \(\mathfrak{t}\). Notice the summation is equal to the weighted Futaki invariant of the orthogonal projection of \(\lambda\) onto \(\mathfrak{t}\) with respect to the inner product \(\langle-,-\rangle_{(\X,\A_T)}\). Thus we can consider this as the Donaldson--Futaki invariant of a test configuration ``orthogonal to \(\mathfrak{t}\)". In fact, letting \(\beta := \sum_{j=1}^r \langle \lambda, \beta_i \rangle_{(\X, \A_T)} \beta_i\), we write \((\X,\A_T)^\perp := (\X,\A_T, -\beta)\), and then \[\DF_{v,w}(\X,\A_T)^\perp = \DF_{v,w}(\X,\A_T,-\beta) = \DF_{v,w}^T(\X,\A_T).\]

\begin{remark}\label{rem:no_beta_needed}
	Given an element \(\beta\in\mathfrak{t}\) we could also define \[\DF^T_{v,w}(\X,\A_T,\beta):=\DF_{v,w}(\X,\A_T,\beta)-\sum_{i=1}^r\langle\lambda+\beta,\beta_i\rangle_{(\X,\A_T)}F_{v,w}(\beta_i),\] however this is equal to \(\DF_{v,w}^T(\X,\A_T)\), by bilinearity of the inner product and linearity of \(F_{v,w}\) on \(\mathfrak{t}\). Thus, \(\DF^T_{v,w}\) is the ``twist invariant" version of \(\DF_{v,w}\).
\end{remark}

\begin{definition}\label{def:weighted_K-stability}
	A \(T\)-equivariant K{\"a}hler manifold \((X,\alpha_T)\) is: \begin{enumerate}
		\item \emph{relatively weighted K-semistable} if \(\DF^T_{v,w}(\X,\A_T)\geq0\) for all test configurations \((\X,\A_T)\) for \((X,\alpha_T)\),
		\item \emph{relatively weighted K-polystable} if it is relatively weighted K-semistable, and \(\DF_{v,w}^T(\X,\A_T)=0\) only if \((\X,\A_T)\) is isomorphic to a product test configuration,
		\item \emph{relatively weighted K-stable} if it is relatively weighted K-polystable, and \(\Aut_0(X)^T=T^{\mathbb{C}}\).
	\end{enumerate}
\end{definition}

\begin{remark}\label{rem:positive_norm}
	For the definition of relative weighted K-polystability, an equivalent condition to being a product test configuration is that the usual \(L^1\)-norm of the projection of the test configuration orthogonal to the torus is positive, by \cite[Appendix]{SD20}. This is further equivalent to positivity of the weighted \(L^1\)-norm. In particular, we will write \(\|(\X,\A_T)^\perp\|_1^w > 0\) in this situation, where \begin{equation}\label{eq:weighted_orthogonal_projection_norm}
	\|(\X,\A_T)^\perp\|_1^w := \int_{\X_0}|(h_\lambda-h_\beta) - (\hat{h}_\lambda - \hat{h}_\beta)|w(m_0) \Omega_0^n,
	\end{equation} where \(\beta\) denotes the orthogonal projection of \(\lambda\) to \(\mathfrak{t}\) via the inner product \(\langle-,-\rangle_{(\X,\A_T)}\) on \(\mathrm{Lie}(S^1\times T)\).
\end{remark}

We remarked that Lahdili gave another definition of relative weighted K-semistability, which we will now cover. By \cite[Section 3.2]{Lah19}, there exists a unique candidate \(\chi\in\mathfrak{t}\) for the weighted extremal field, regardless of whether a weighted extremal metric exists. This is obtained by taking the \(L^2\)-orthogonal projection of \(S_{v,w}(\omega)\) onto the hamiltonian generators of \(\mathfrak{t}\), with respect to the inner product defined by the measure \(\frac{1}{n!}w(\mu)\omega^n\); this projection generates a unique \(\chi\in\mathfrak{t}\), independent of the choice of representative \(\omega+\mu \in \alpha_T\). The weighted extremal equation is then \[S_{v,w}(\omega)=\mu^\chi+a\] for a suitable constant \(a\in\mathbb{R}\) depending only on the equivariant cohomology class \([\omega+\mu]\), the weights \(v,w\), and \(\chi\in\mathfrak{t}\). From this observation, one can think of a weighted extremal metric as a solution to the equation \[S_v(\omega)=w(\mu)w_{\mathrm{ext}}(\mu),\] where \(w_{\mathrm{ext}}(p) := \langle p,\chi\rangle + a\). In particular, we may define \(w':=ww_{\mathrm{ext}}\). Even though \(w'\) is not positive, we may still define \(\DF_{v,w'}(\X,\A_T)\) when \(\X\) is smooth via the formula \begin{align}\label{eq:DF_not_positive}
	\DF_{v,w'}(\X,\A_T) :=&-\frac{1}{(n+1)!}\int_{\X}\left(S_v(\Omega)-\hat{S}_{v,w'}w'(m_T)\right)\Omega^{n+1} \\
	&+\frac{2}{n!}\int_{\X}v(m_T)\pi^*\omega_{\mathrm{FS}}\wedge\Omega^n,\nonumber
\end{align} where \(\omega_{\mathrm{FS}}\) denotes the Fubini--Study metric on \(\mathbb{P}^1\) \cite[Definition 11]{Lah19}.\footnote{In the case where \(\int_Xw'(\mu)\omega^n = 0\), we simply replace \(\hat{S}_{v,w'}\) with \(1\) in the definition of \(\DF_{v,w'}(\X,\A_T)\); see \cite[Definitions 3, 11]{Lah19}.}

Lahdili then shows that a weighted extremal manifold satisfies \[\DF_{v,w'}(\X,\A_T)\geq0\] for all smooth test configurations \((\X,\A_T)\). We shall show that this result is equivalent to relative weighted K-semistability of weighted extremal manifolds, in the sense of Definition \ref{def:weighted_K-stability}:

\begin{theorem}[{\cite[Corollary 2]{Lah20b}}]\label{thm:weighted_semistable}
	Let \((X,\alpha_T)\) be a weighted extremal manifold. Then \((X,\alpha_T)\) is relatively weighted K-semistable.
\end{theorem}

To prove this, we first show the following:

\begin{lemma}
	Let \((X,\A_T)\) be a smooth test configuration. Then \[\DF_{v,ww_{\mathrm{ext}}}(\X,\A_T) = \DF^T_{v,w}(\X,\A_T),\] where \(\DF_{v,ww_{\ext}}(\X,\A_T)\) is defined in equation \ref{eq:DF_not_positive}.
\end{lemma}

\begin{proof}
	Recall that the extremal field \(w_{\ext}\) is defined by the \(L^2\)-orthogonal projection of \(S_{v,w}(\omega)\) onto the generators of \(\mathfrak{t}\), with respect to the inner product defined by the measure \(\frac{1}{n!}w(\mu)\omega^n\). It follows that the weighted Futaki invariant is given by \begin{align*}
		F_{v,w}(\beta) &= \frac{1}{n!}\int_X(\hat{S}_{v,w}-S_{v,w}(\omega))h_\beta w(\mu)\omega^n \\
		&=-\frac{1}{n!}\int_X (w_{\ext}(\mu)-\hat{S}_{v,w})(h_\beta - \hat{h}_\beta) w(\mu)\omega^n \\
		&=-\langle w_{\mathrm{ext}},\beta\rangle.
	\end{align*} Now, in the definition of \(\DF^T_{v,w}(\X,\A_T)\), we compute the weighted Futaki invariant \(F_{v,w}(\beta_i)\) of a basis element \(\beta_i\). In the proof of Lemma \ref{lem:inner_product} we observed that the inner product of elements of \(\mathfrak{t}\) can be computed on the central fibre of an arbitrary test configuration. Thus, we have \begin{align*}
		\DF_{v,w}^T(\X,\A_T) &= \DF_{v,w}(\X,\A_T) - \sum_j\langle \lambda,\beta_j\rangle_{(\X,\A_T)} F_{v,w}(\beta_j) \\
		&= \DF_{v,w}(\X,\A_T) + \sum_j\langle \lambda,\beta_j\rangle_{(\X,\A_T)} \langle w_{\mathrm{ext}},\beta_j\rangle_{(\X,\A_T)} \\
		&= \DF_{v,w}(\X,\A_T) + \langle \lambda,w_{\mathrm{ext}}\rangle_{(\X,\A_T)},
	\end{align*} where \(\lambda\) generates the \(\mathbb{C}^*\)-action on the test configuration. It remains to show that this final line is equal to \(\DF_{v,ww_{\mathrm{ext}}}(\X,\A_T)\). To see this, we use an asymptotic slope formula due to Lahdili for the functional \(\mathcal{E}_u(\varphi)\) defined by its first variation along smooth paths: \[\frac{d}{dt}\mathcal{E}_u(\varphi_t) := \int_X \dot{\varphi}_t u(\mu_t)\omega_t^n,\] and normalised so that \(\E_u(0)=0\); here \(u:P\to\mathbb{R}\) is an arbitrary smooth function on the moment polytope. On \(X\cong\X_1\) we write \(\omega_t := \omega + i\d\db\varphi_t = \lambda(e^{-t})^*\Omega|_{\X_1}\), and write \(\mu_t\) for the corresponding moment map induced by \(m_T\). By \cite[Lemma 9]{Lah19},  \[\lim_{t\to\infty}\frac{d}{dt}\E_u(\varphi_t)=\frac{1}{(n+1)!}\int_{\X}u(m_T)\Omega^{n+1}.\] In particular, using equation \eqref{eq:DF_not_positive}, \begin{align*}
		\DF_{v,ww_{\mathrm{ext}}}(\X,\A_T) - \DF_{v,w}(\X,\A_T) &= \frac{1}{(n+1)!}\int_{\X}(w_{\mathrm{ext}}(m_T)-\hat{S}_{v,w})w(m_T)\Omega^{n+1} \\
		&= \lim_{t\to\infty}\frac{d}{dt}\E_{(w_{\mathrm{ext}}-\hat{S}_{v,w})w}(\varphi_t) \\
		&= \lim_{t\to\infty}\int_X\dot{\varphi}_t(w_{\ext}(\mu_t)-\hat{S}_{v,w})w(\mu_t)\omega_t^n \\
		&= \int_{\X_0}h_\lambda(w_{\ext}(m_T)-\hat{S}_{v,w})w(m_0)\Omega_0^n \\
		&= \langle \lambda, w_{\mathrm{ext}}\rangle_{(\X,\A_T)}. 
	\end{align*} Thus, \(\DF_{v,w'}(\X,\A_T)\) and \(\DF^T_{v,w}(\X,\A_T)\) are both computed as \(\DF_{v,w}(\X,\A_T)+\langle \lambda, w_{\ext} \rangle_{(\X,\A_T)}\), and we are done.
\end{proof}

\begin{proof}[Proof of Theorem \ref{thm:weighted_semistable}]

Let \((\X,\A_T)\) be a test configuration for \((X,\alpha_T)\). Take an equivariant resolution of singularities \(\widetilde{\X}\to\X\), let \(E_T\subset\widetilde{\X}\) be the exceptional divisor of the resolution, and consider \((\widetilde{\X},\widetilde{\A}_T - \epsilon E_T)\), which is a smooth test configuration for \((X,\alpha_T)\) for all \(\epsilon > 0\) sufficiently small. Sending \(\epsilon\to0\), we observe \[\DF_{v,w}^T(\widetilde{\X},\widetilde{\A}_T - \epsilon E_T) \to \DF^T_{v,w}(\X,\A_T).\] Since \(\DF^T_{v,w}(\widetilde{\X}, \widetilde{\A}_T - \epsilon E_T)\geq 0\) for all \(\epsilon > 0\) sufficiently small, it follows that \(\DF_{v,w}^T(\X,\A_T)\geq0\).

\end{proof}

	
\section{Relative K-polystability of weighted extremal manifolds}

In this section we prove Theorem \ref{thm:main}, that a weighted extremal manifold is relatively weighted K-polystable. The style of argument we use was coined by Stoppa for polarised cscK manifolds with discrete automorphism group \cite{Sto09,Sto11}. It was then generalised to the setting of polarised extremal manifolds by Stoppa--Sz{\'e}kelyhidi \cite{SS11}. In the non-polarised K{\"a}hler setting, the argument was extended by Dervan--Ross to the cscK case \cite{DR17}, then later by Dervan to the extremal case \cite{Der18}.

Before beginning the proof of the main theorem, we first assume it holds to prove Corollary \ref{cor:main}, that a weighted cscK manifold is weighted K-polystable:

\begin{proof}[Proof of Corollary \ref{cor:main}]
	Let \((X,\alpha_T)\) be a weighted cscK manifold. Then by \cite[Corollary 2]{Lah20b} and \cite[Theorem A]{Ino20}, the manifold \((X,\alpha_T)\) is weighted K-semistable. Let \((\X,\A_T)\) be a test configuration for \((X,\alpha_T)\), and \(\beta\in\mathfrak{t}\) be such that \(\DF_{v,w}(\X,\A_T,\beta)=0\). Since \((X,\alpha_T)\) is weighted cscK we have \(F_{v,w}=0\) on \(\mathfrak{t}\), and so \(\DF_{v,w}(\X,\A_T,\beta)=\DF_{v,w}^T(\X,\A_T)=0\). Being weighted cscK, \((X,\alpha_T)\) is in particular weighted extremal, and so by Theorem \ref{thm:main} is relatively weighted K-polystable. It follows that \((\X,\A_T)\) is a product test configuration. 
\end{proof}

\subsection{Expansions of weighted invariants}

To prove the main theorem, we will find a suitable \(T\)-invariant point \(p\in X\), blow up the test configuration \((\X,\A_T)\) along the orbit closure \(C:=\overline{\mathbb{C}^*p}\subset\X\) of \(p\), and compute the expansion of the \(\DF_{v,w}^T\)-invariant of the blown up test configuration as \(\epsilon\to0\), where \(\epsilon\) is the coefficient of the exceptional divisor. When expanding the unweighted Donaldson--Futaki invariant, the subleading coefficient is given by the Chow weight \[\Ch_p(\X,\A) := \frac{1}{n+1}\frac{\A^{n+1}}{\alpha^n} - \int_C\Omega,\] see \cite[Proposition 5.4]{DR17}. In our setting, it will be given by the following weighted analogue:

\begin{definition}
	Let \((\X,\A_T)\) be a test configuration for \((X,\alpha_T)\). Let \(p\in X\) be a \(T\)-invariant point, and define \(C:=\overline{\mathbb{C}^*p}\subset\X\). The \emph{weighted Chow weight of \(p\)} is \[\mathrm{Ch}_p^w(\X,\A_T):=\frac{(g(\A_T))(\xi)}{(g'(\alpha_T))(\xi)}-\int_C\Omega.\]
\end{definition}

We then have the following expansion of the weighted Donaldson--Futaki invariant. We first consider the case where both the test configuration and the curve \(C\) are smooth for notational simplicity, and later explain how to compute the expansion in the general case.

\begin{proposition}\label{prop:DF_expansion_smooth}
	Let \((\X,\A_T)\) be a smooth test configuration for \((X,\alpha_T)\), and let \(p\in X\) be a \(T\)-fixed point. Denote by \(C\) the orbit closure of \(p\in\X_1\cong X\), and let \((\Bl_C\X,\widetilde{\A}_T-\epsilon\E_T)\) be the blown up test configuration with \(T\)-invariant exceptional divisor \(\E_T\), depending on a parameter \(0<\epsilon\ll1\); here \(\widetilde{\A}_T\) denotes the pullback of \(\A_T\) to \(\Bl_C\X\). Suppose that \(C\) is smooth. Then \begin{align*}
	&\DF_{v,w}(\Bl_C\X,\widetilde{\A}_T-\epsilon\E_T) \\
	=&\DF_{v,w}(\X,\A_T)- \frac{v(\mu(p))}{(n-2)!}\mathrm{Ch}_p^w(\X,\A_T)\epsilon^{n-1}+\O(\epsilon^n).
	\end{align*}
\end{proposition}

\begin{proof}
	Write \[f(x) = \sum_k \frac{a_k}{k!} x^k;\quad\quad g(x) = \sum_k \frac{b_k}{k!} x^k,\] so that \[f^{(\ell)}(x) = \sum_{k=0}^\infty \frac{a_{k+\ell}}{k!} x^k;\quad\quad g^{(\ell)}(x) = \sum_{k=0}^\infty \frac{b_{k+\ell}}{k!} x^k.\] We also write \(\widetilde{\alpha}_T\) for the pullback of \(\alpha_T\) to \(\widetilde{X}:=\Bl_pX\), and \(\widetilde{\A}_T\) for the pullback of \(\A_T\) to \(\widetilde{\X}:=\Bl_C\X\). Denote by \(E_T\) the exceptional divisor of \(\Bl_pX\) and \(\E_T\) the exceptional divisor of \(\Bl_C\X\). Then \begin{align*}
		g'(\widetilde{\alpha}_T-\epsilon E_T) &= g'(\widetilde{\alpha}_T) + (-\epsilon E_T) \sum_{k=1}^\infty \frac{b_{k+1}}{k!} \binom{k}{1} \widetilde{\alpha}_T^{k-1}
		+ (-\epsilon E_T)^2\sum_{k=2}^\infty \frac{b_{k+1}}{k!}\binom{k}{2}\widetilde{\alpha}_T^{k-2} +\cdots \\
		&= g'(\widetilde{\alpha}_T)+(-\epsilon E_T)\sum_{k=0}^\infty \frac{b_{k+2}}{k!} \widetilde{\alpha}_T^k +\frac{1}{2!}(-\epsilon E_T)^2\sum_{k=0}^\infty \frac{b_{k+3}}{k!}\widetilde{\alpha}_T^k+\cdots \\
		&= g'(\widetilde{\alpha}_T)+(-\epsilon E_T)g^{(2)}(\widetilde{\alpha}_T)+\frac{1}{2!}(-\epsilon E_T)^2g^{(3)}(\widetilde{\alpha}_T)+\cdots.
	\end{align*} Now, note that \(\widetilde{\alpha}_T|_{E_T} = [0+\mu(p)]\) is constant, so \[((-E_T)^k\widetilde{\alpha}_T^\ell)=0\] whenever \(k<n\) and \[((-E_T)^n\widetilde{\alpha}_T^\ell) = \mu(p)^\ell((-E)^n)=-\mu(p)^\ell\] for \(\ell\geq1\), where the intersection \(((-E)^n)\) is computed in the usual non-equivariant cohomology ring. Hence, \begin{align*}
		(g'(\widetilde{\alpha}_T-\epsilon E_T))(\xi) &= (g'(\widetilde{\alpha}_T))(\xi) + \frac{1}{n!}((-\epsilon E_T)^ng^{(n+1)}(\widetilde{\alpha}_T))(\xi) + \O(\epsilon^{n+1}) \\
		&= (g'(\alpha_T))(\xi) - \frac{w(\mu(p))}{n!} \epsilon^n + \O(\epsilon^{n+1}).
	\end{align*} Next, note that \(\widetilde{\A}_T|_{\E_T}\) is the pullback of \(\A_T|_{C}\) under the map \(\E_T\to C\). It follows that \[((-\E_T)^k\widetilde{\A}_T^\ell)=0\] whenever \(k\leq n-1\), and \[((-\E_T)^n\widetilde{\A}_T^\ell) = -\ell \int_Cm_T^{\ell-1} \Omega\] for \(\ell\geq1\), noting that \(\A_T=[\Omega+m_T]\). But note the \(T\)-action on \(C\) is trivial, hence \(m_T\) is constant on \(C\) (equal to \(\mu(p)\)) and can be pulled out of the integral. Thus, similar to the calculation on \(\Bl_pX\), \[(g(\widetilde{\A}_T-\epsilon \E_T))(\xi) = (g(\A_T))(\xi) - \frac{w(\mu(p))\int_C\Omega}{n!}\epsilon^n + \O(\epsilon^{n+1}).\]
	
	Next we compute \((c_1(\widetilde{X})_T f'(\widetilde{\alpha}_T - \epsilon E_T))\). Similarly to \(g\), \[f'(\widetilde{\alpha}_T-\epsilon E_T) = f'(\widetilde{\alpha}_T) + (-\epsilon E_T)f^{(2)}(\widetilde{\alpha}_T)+\frac{1}{2!}(-\epsilon E_T)^2f^{(3)}(\widetilde{\alpha}_T)+\cdots.\] Next note \(c_1(\widetilde{X})_T = \widetilde{c}_1(X)_T - (n-1)[E_T]\), where we write \(\widetilde{c}_1(X)_T\) for the pull back of \(c_1(X)_T\) to \(\widetilde{X}\). Since \(\widetilde{c}_1(X)_T|_{E_T} = [0+\Delta\mu (p)]\) is constant, \[(\widetilde{c}_1(X)_T f'(\widetilde{\alpha}_T-\epsilon E_T))(\xi) = (c_1(X)_T f'(\alpha_T))(\xi) + \O(\epsilon^n).\] Next, \begin{align*}
		(E_Tf'(\widetilde{\alpha}_T-\epsilon E_T)) (\xi) &= \frac{1}{(n-1)!}\epsilon^{n-1}(E_T(-E_T)^{n-1}f^{(n)}(\widetilde{\alpha}_T))(\xi) + \O(\epsilon^n) \\
		&= \frac{v(\mu(p))}{(n-1)!}\epsilon^{n-1}+\O(\epsilon^n).
	\end{align*} Hence \[(\widetilde{c}_1(X)_T f'(\widetilde{\alpha}_T - \epsilon E_T))(\xi) = (c_1(X)_T f'(\alpha_T))(\xi) - \frac{v(\mu(p))}{(n-2)!} \epsilon^{n-1} + \O(\epsilon^n).\] In the same manner we readily compute \begin{align*}
	& (c_1(\widetilde{\X}/\mathbb{P}^1)_T f(\widetilde{\A}_T - \epsilon \E_T))(\xi) \\
	=& (c_1(\X/\mathbb{P}^1)_T f(\A_T))(\xi) - \frac{v(\mu(p)) \int_C\Omega}{(n-2)!} \epsilon^{n-1} + \O(\epsilon^{n+1}),
	\end{align*} and we will give more details of this in the more general singular case below. From these calculations, the expansion of the weighted Donaldson--Futaki invariant follows immediately.
\end{proof}

When \(\X\) and \(C\) are not necessarily smooth, we use the following modification:

\begin{proposition}\label{prop:blowup_general}
	Let \((\X,\A_T)\) be a test configuration for \((X,\alpha_T)\), let \(p\in X\) be a \(T\)-fixed point, and define \(C:=\overline{\mathbb{C}^*p}\subset\X\). Let \(r:\mathcal{Y}\to\X\) be a \(T^{\mathbb{C}}\times\mathbb{C}^*\)-equivariant resolution of singularities of \(\X\) with exceptional divisor \(\mathcal{F}\) supported on \(\mathcal{Y}_0\), such that \(\mathcal{Y}_0\) is a simple normal crossings divisor\footnote{The assumption that \(\mathcal{Y}_0\) is a simple normal crossings divisor is not needed for this particular expansion, however it will be needed for later results.} and the proper transform \(\hat{C}\subset\mathcal{Y}\) of \(C\) is smooth. Let \(b:\mathcal{B}\to\mathcal{Y}\) be the blowup of \(\mathcal{Y}\) along \(\hat{C}\), with exceptional divisor \(\E\). Then the weighted Donaldson--Futaki invariant of the test configuration \((\mathcal{B},b^*r^*\A_T-\epsilon\E_T-\epsilon^n b^*\mathcal{F}_T)\) has an expansion  \begin{align*}
		&\DF_{v,w}(\mathcal{B},b^*r^*\A_T-\epsilon\E_T-\epsilon^n b^*\mathcal{F}_T) \\
		=&\DF_{v,w}(\X,\A_T) - \frac{v(\mu(p))}{(n-2)!}\mathrm{Ch}_p^w(\X,\A_T)\epsilon^{n-1}+\O(\epsilon^n).
	\end{align*}
\end{proposition}

\begin{proof}
	The same calculation as in the proof of Proposition \ref{prop:DF_expansion_smooth} carries through, only it is more notationally cumbersome. First, note that \(\epsilon^n b^*\mathcal{F}_T\) can be effectively ignored in the calculation, since any term involving it will be \(\O(\epsilon^n)\). Hence the calculation of most terms is essentially unchanged; we will however calculate the term \((c_1(\mathcal{B}/\mathbb{P}^1)_T f(b^*r^*\A_T - \epsilon\E_T-\epsilon^n b^*\mathcal{F}_T))(\xi)\).
	
	First, note \(c_1(\mathcal{B}/\mathbb{P}^1)_T = b^*c_1(\mathcal{Y}/\mathbb{P}^1)_T - (n-1)[\E_T]\). For \(f(b^*r^*\A_T - \epsilon \E_T - \epsilon^n b^* \mathcal{F}_T)\), ignoring \(\epsilon^n b^* \mathcal{F}_T\) we compute \begin{align*}
		f(b^*r^*\A_T - \epsilon \E_T) &= f(b^*r^*\A_T) + \frac{1}{1!} (-\epsilon \E_T) \sum_{k=0}^\infty \frac{a_{k+1}}{k!}(b^*r^*\A_T)^k + \cdots  \\
		&+ \frac{1}{(n-1)!} (-\epsilon \E_T)^{n-1} \sum_{k=0}^\infty \frac{a_{k+n-1}}{k!}(b^*r^*\A_T)^k +\O(\epsilon^n).
	\end{align*} Note that \((-\E_T)^\ell b^*r^*\A_T^k\) has trivial integral whenever \(\ell<n\), since \(b^*r^*\A_T|_{\E_T}\) is pulled back from \(\hat{C}\). Furthermore, since \(\A_T = [\Omega + m_T]\) and \(r^* m_T|_{\hat{C}} = \mu(p)\) is constant, for \(k\geq1\) we have \[((-\E_T)^n b^*r^*\A_T^k)=-k\mu(p)^{k-1}\int_{\hat{C}}r^*\Omega.\] Therefore \begin{align*}
		&((n-1)\E_T f(b^*r^*\A_T-\epsilon\E_T))(\xi) \\
		=& \epsilon^{n-1}\frac{n-1}{(n-1)!}\sum_{k=0}^\infty\frac{a_{k+n-1}}{k!}(\E_T(-\E_T)^{n-1}(b^*r^*\A_T)^k)(\xi) + \O(\epsilon^n) \\
		=& \epsilon^{n-1}\frac{1}{(n-2)!}\sum_{k=0}^\infty\frac{a_{k+n}}{k!}\langle\mu(p),\xi\rangle^k\int_{\hat{C}}r^*\Omega + \O(\epsilon^n) \\
		=&
		\frac{v(\mu(p))\int_{C}\Omega}{(n-2)!}\epsilon^{n-1} + \O(\epsilon^n),
	\end{align*} where we used \(\int_{\hat{C}}r^*\Omega = \int_C\Omega\), which can be seen by integrating over the smooth locus of \(C\). Next, \begin{align*}
		(b^*c_1(\mathcal{Y}/\mathbb{P}^1)_T f(b^*r^*\A_T - \epsilon \E_T))(\xi) &= (b^*c_1(\mathcal{Y}/\mathbb{P}^1)_T f(b^*r^*\A_T))(\xi) + \O(\epsilon^n) \\
		&= (c_1(\mathcal{Y}/\mathbb{P}^1)_T f(r^*\A_T))(\xi) + \O(\epsilon^n).
	\end{align*} Here we used the projection formula for \(b:\mathcal{B}\to\mathcal{Y}\) to reduce the first term. The higher order terms are order \(\epsilon^n\), since \(b^*c_1(\mathcal{Y}/\mathbb{P}^1)_T|_{\E_T}\) and \(b^*r^*\A_T|_{\E_T}\) are pulled back from \(\hat{C}\), which is 1-dimensional. Combining these computations, we deduce \begin{align*}
		&(c_1(\mathcal{B}/\mathbb{P}^1)_T f(b^*r^*\A_T - \epsilon\E_T-\epsilon^n b^*\mathcal{F}_T))(\xi) \\
		=& (c_1(\mathcal{Y}/\mathbb{P}^1)_T f(r^*\A_T))(\xi) - \frac{v(\mu(p))\int_{C}\Omega}{(n-2)!}\epsilon^{n-1} + \O(\epsilon^n).
	\end{align*} Recall that the intersections in the weighted Donaldson--Futaki invariant are defined by pulling back to a resolution of singularities, so this first term is the appropriate one appearing in the weighted Donaldson--Futaki invariant of \((\X,\A_T)\).
\end{proof}

Since we wish to compute relative stability, we must also compute how the inner products \(\langle-,-\rangle_{(\X,\A_T)}\) and the weighted Futaki invariants \(F_{v,w}\) change on the blowup. First, we state the following analogue of \cite[Proposition 37]{Sze15} in the weighted setting:

\begin{lemma}\label{lem:T-ip_blowup}
	Let \((X,\alpha_T)\) be a K{\"a}hler manifold with \(T\)-action, and let \(p\in X\) be a \(T\)-invariant point. Denote by \(\langle-,-\rangle\) the inner product on \(\mathfrak{t}\) defined on the base manifold \((X,\alpha_T)\), and by \(\langle-,-\rangle_\epsilon\) the inner product on \(\mathfrak{t}\) defined on \((\Bl_pX,\alpha_T - \epsilon E_T)\). Then \[\langle\beta_1, \beta_2\rangle_\epsilon = \langle\beta_1, \beta_2\rangle + \O(\epsilon^{n-\delta})\] for all \(\delta > 0\) sufficiently small.
\end{lemma}

The exact same proof from \cite{Sze15} carries over with minimal modification, since the measure \(w(\mu)\omega^n\) is equivalent to the measure \(\omega^n\). We only note that the result in \cite{Sze15} is stated in the case \(\langle\beta_1,\beta_2\rangle = 0\), but the proof works equally well without this assumption.

\begin{lemma}\label{lem:blowup_inner_product}
	Given the setup of Proposition \ref{prop:blowup_general}, denote by \(\lambda\) the generator of the \(\mathbb{C}^*\)-action of the test configuration \(\X\), and let \(\beta\in\mathfrak{t}\). Then \begin{align*}
		\langle \lambda,\beta\rangle_{(\mathcal{B}, b^*r^*\A_T-\epsilon \E_T-\epsilon^{n}b^*E_{\mathcal{Y},T})} = \langle \lambda, \beta \rangle_{(\X,\A_T)} + \O(\epsilon^{n-\delta}),
	\end{align*} for all \(\delta>0\) sufficiently small.
\end{lemma}

The proof is unchanged from \cite[Proposition 4.16]{Der18}, where one uses the fact that \(\mathcal{Y}_0\) is simple normal crossings divisor to reduce the computation to the smooth case, and apply Lemma \ref{lem:T-ip_blowup}.

\begin{lemma}\label{lem:Futaki_blowup}
	Let \((X,\alpha_T)\) be a K{\"a}hler manifold with \(T\)-action, and let \(p\in X\) a \(T\)-invariant point. Denote by \(F_{v,w}^\epsilon\) the weighted Futaki invariant on the blown up manifold \(\Bl_pX\) for the class \(\alpha_T-\epsilon E_T\), where \(E\) is the exceptional divisor of the blowup. Then weighted Futaki invariant has an expansion \[F_{v,w}^\epsilon(\beta) = F_{v,w}(\beta) + \frac{v(\mu(p))}{(n-2)!}(h_\beta(p) - \hat{h}_\beta)\epsilon^{n-1} + \O(\epsilon^n).\]
\end{lemma}

\begin{proof}
	 This is a straightforward consequence of Proposition \ref{prop:blowup_general}.\footnote{I thank Ruadha{\'i} Dervan for pointing this out to me.} In particular, in the case \(\beta\in\mathfrak{t}\) is integral, we let \((\X,\A_T)\) be the product test configuration associated to \(\beta\). By Lemma \ref{lem:twisted_Chow_weight} below, the weighted Chow weight is computed as \(h_\beta(p)-\hat{h}_\beta\). The weighted Donaldson--Futaki invariant of a product test configuration is merely the weighted Futaki invariant of the corresponding 1-parameter subgroup, so the expansion follows immediately. For a general element \(\beta\in\mathfrak{t}\), we write \(\beta\) as a linear combination of integral elements, and apply the expansion in the integral case.
\end{proof}

We will see that in the expansion of \(\DF_{v,w}^T\), rather than the weighted Chow weight of a point we get a \(T\)-orthogonal analogue; that is, a variant of the weighted Chow weight that is invariant under twisting the test configuration by one-parameter subgroups of \(T\).

\begin{definition}
	The \emph{\(T\)-orthogonal weighted Chow weight} is \[\mathrm{Ch}_p^T(\X,\A_T):=\frac{(g(\A_T))(\xi)}{(g'(\alpha_T))(\xi)}-\int_C\Omega+\sum_{i=1}^r\langle\lambda,\beta_i\rangle_{(\X,\A_T)}(h_{\beta_i}(p)-\hat{h}_{\beta_i}).\]
\end{definition}

We will study this invariant in more detail in the following subsection. For now, we describe the expansion of \(\DF_{v,w}^T\) in terms of the \(T\)-orthogonal Chow weight. This is an immediate consequence of the definition of \(\DF^T_{v,w}\) and the expansions of Proposition \ref{prop:blowup_general}, Lemma \ref{lem:T-ip_blowup}, Lemma \ref{lem:blowup_inner_product}, and Lemma \ref{lem:Futaki_blowup}.

\begin{proposition}\label{prop:DFT_blowup}
	Let \((\X,\A_T)\) be a test configuration for \((X,\alpha_T)\), let \(p\in X\) be a \(T\)-fixed point, and define \(C:=\overline{\mathbb{C}^*p}\subset\X\). Let \(r:\mathcal{Y}\to\X\) be an equivariant resolution of singularities of \(\X\) with exceptional divisor \(\mathcal{F}_T\) supported on \(\mathcal{Y}_0\), such that \(\mathcal{Y}_0\) is a simple normal crossings divisor and the proper transform \(\hat{C}\subset\mathcal{Y}\) of \(C\) is smooth. Let \(b:\mathcal{B}\to\mathcal{Y}\) be the blowup of \(\mathcal{Y}\) along \(\hat{C}\), with exceptional divisor \(\E_T\). Then the \(\DF^T_{v,w}\)-invariant of the test configuration \((\mathcal{B},b^*r^*\A_T-\epsilon\E_T-\epsilon^nb^*\mathcal{F}_T)\) for \((\Bl_pX,\alpha_T-\epsilon E_T)\) has an expansion  \begin{align*}
			&\DF_{v,w}^T(\mathcal{B},b^*r^*\A_T-\epsilon\E_T-\epsilon^{n}b^*\mathcal{F}_T) \\
			=&\DF_{v,w}^T(\X,\A_T) - \frac{v(\mu(p))}{(n-2)!} \mathrm{Ch}_p^T(\X,\A_T) \epsilon^{n-1} + \O(\epsilon^\kappa),
		\end{align*} where \(\kappa > n-1\).
\end{proposition}

The last thing we require is the existence of a \(T\)-invariant point with positive weighted \(T\)-Chow weight. We state the result here, and prove it in the next subsection.

\begin{proposition}\label{prop:Chow}
	Let \((\X,\A_T)\) be a \(T\)-equivariant test configuration for \((\X,\alpha_T)\) with \(\|(\X,\A_T)^\perp\|^w_1 > 0\), where \((\X,\A_T)^\perp\) denotes the component of the test configuration orthogonal to the torus and the norm is defined in equation \eqref{eq:weighted_orthogonal_projection_norm}. Then there exists a \(T\)-invariant point \(p\in X\) such that \[\Ch_p^T(\X,\A_T) > 0.\] 
\end{proposition}

Given all this, we can finally complete the proof of Theorem \ref{thm:main}:

\begin{proof}[Proof of Theorem \ref{thm:main}]
	Let \((X,\alpha_T)\) be a weighted extremal manifold. Then \((X,\alpha_T)\) is relatively weighted K-semistable by Theorem \ref{thm:weighted_semistable}. Suppose that \((\X,\A_T)\) is a test configuration for \((X,\alpha_T)\) which satisfies \[\DF_{v,w}^T(\X,\A_T)=0.\] We wish to show that this test configuration must be a product test configuration. If it is not a product test configuration, then by Remark \ref{rem:positive_norm}, we have \[\|(\X,\A_T)^\perp\|^w_1 > 0.\] By Proposition \ref{prop:Chow}, there exists a \(T\)-fixed point \(p\in X\) with \[\Ch_p^T(\X,\A_T) > 0.\] Define \(C:=\overline{\mathbb{C}^*p}\subset \X\). Let \(r:\mathcal{Y}\to\X\) be an equivariant resolution of singularities of \(\X\) with exceptional divisor \(\mathcal{F}_T\) supported on \(\mathcal{Y}_0\), such that \(\mathcal{Y}_0\) is a simple normal crossings divisor and the proper transform \(\hat{C}\subset\mathcal{Y}\) of \(C\) is smooth. Let \(b:\mathcal{B}\to\mathcal{Y}\) be the blowup of \(\mathcal{Y}\) along \(\hat{C}\), with exceptional divisor \(\E_T\). By Proposition \ref{prop:DFT_blowup}, \begin{align*}
		&\DF_{v,w}^T(\mathcal{B},b^*r^*\A_T-\epsilon\E_T-\epsilon^nb^*\mathcal{F}_T) \\
		=&\DF_{v,w}^T(\X,\A_T)- \frac{v(\mu(p))}{(n-2)!}\mathrm{Ch}_p^T(\X,\A_T)\epsilon^{n-1}+\O(\epsilon^\kappa),
	\end{align*} for some \(\kappa > n-1\). Since the weighted Chow weight of \(p\) is strictly positive, it follows that \[\DF_{v,w}^T(\mathcal{B},b^*r^*\A_T-\epsilon\E_T-\epsilon^nb^*\mathcal{F}_T)<0\] for \(\epsilon>0\) sufficiently small. Since \(T\) is maximal, the point \(p\) is relatively stable, so by \cite[Theorem 1.1]{Hal23}, \((\Bl_pX,\alpha_T-\epsilon E_T)\) admits a weighted extremal metric for all \(\epsilon>0\) sufficiently small. It is therefore relatively weighted K-semistable, and we have reached a contradiction.
\end{proof}
		
\subsection{Existence of a Chow stable point}

In this final section we prove Proposition \ref{prop:Chow}, which was used in the proof of Theorem \ref{thm:main}.

\begin{lemma} 
	For \(a\in\mathbb{R}\), \[\mathrm{Ch}_p(\X,\A_T+a\mathcal{O}_{\mathbb{P}^1}(1))=\mathrm{Ch}_p(\X,\A_T).\]
\end{lemma}

\begin{proof}
	The restriction of \(\A_T+a\mathcal{O}_{\mathbb{P}^1}(1)\) to the general fibre is still \(\alpha_T\), and so the denominator \((g'(\alpha_T))(\xi)\) is unchanged by this perturbation. Since \(C\to\mathbb{P}^1\) is a bihomolorphism away from \(0\in\mathbb{P}^1\), \[\int_C(\Omega+a\pi^*\omega_{\mathrm{FS}}) = a + \int_C\Omega.\] On the other hand, \begin{align*}
		(g(\A_T+a\mathcal{O}_{\mathbb{P}^1}(1)))(\xi) &= \frac{1}{(n+1)!}\int_{\X}w(m_T)(\Omega+a\pi^*\omega_{\mathrm{FS}})^{n+1} \\
		&= \frac{1}{(n+1)!}\int_{\X}w(m_T)\Omega^{n+1}+\frac{a}{n!}\int_{\X}w(m_T)\Omega^n\wedge\pi^*\omega_{\mathrm{FS}} \\
		&= (g(\A_T))(\xi) + \frac{a}{n!} \int_{\X_1}w(m_1)\Omega_1^n \\
		&= (g(\A_T))(\xi) + a (g'(\alpha_T))(\xi). \qedhere
	\end{align*}
\end{proof}

It follows we may normalise \(\A_T\) so that \((g(\A_T))(\xi)=0\), in which case \(\Ch_p(\X,\A_T)=-\int_C\Omega\). Note that in the non-weighted setting, the usual Chow weight can also be normalised as such, although with a potentially different normalisation so we cannot say that the weighted and unweighted Chow weights coincide.

In order to prove the existence of a \(T\)-invariant point \(p\) with \(\Ch_p^T(\X,\A_T) > 0\), we will first prove a kind of uniform Chow stability, in terms of the weighted \(L^1\)-norm of the test configuration. To do this, we will describe the weighted \(L^1\)-norm in terms of the weak geodesic ray associated to the test configuration; see \cite{Che00, PS07} and \cite[Section 2.4]{Ber16} for generalities on geodesic rays. Such a description was first found in the unweighted projective setting by Hisamoto \cite[Theorem 1.2]{His16}.

\begin{lemma}\label{lem:integral_limits}
	Let \((\X,\A_T)\) be a test configuration for \((X,\alpha_T)\), and let \(\psi_t\) be the associated weak geodesic ray. Then \[\|(\X,\A_T)\|_1^w = \lim_{t\to\infty}\int_X|\dot{\psi}_t|w(\mu_t)\omega_t^n,\] where \(\omega_t := \omega + i\d\db\psi_t\) and \(\mu_t := \mu + d^c\psi_t\) is the associated moment map. Similarly \[(g(\A_T))(\xi) = \lim_{t\to\infty}\int_X\dot{\psi}_t w(\mu_t) \omega_t^n.\]
\end{lemma}

\begin{proof}
	Note that for any continuous function \(f:\mathbb{R}\to\mathbb{R}\), \[\lim_{t\to\infty}\int_Xf(\dot{\psi}_t) w(\mu_t)\omega_t^n = \lim_{\tau\to0}\int_{\X_\tau}f(h_{\Psi}) w(m_{T,\Psi}) \Omega_{\Psi,\tau}^n = \int_{\X_0}f(h_{\Psi}) w(m_{T,\Psi}) \Omega_{\Psi,0}^n,\] where \(\Omega_\Psi = \Omega + i\d\db\Psi\) is the solution to the geodesic equation of the test configuration, \(h_\Psi := h_\lambda + d^c\Psi\) is the hamiltonian for the \(\mathbb{C}^*\)-action with respect to \(\Omega_\Psi\), and \(m_{T,\Psi}\) the \(T\)-moment map with respect to \(\Omega_\Psi\). We therefore wish to show that \[\int_{\X_0}f(h_{\Psi}) w(m_{T,\Psi}) \Omega_{\Psi,0}^n = \int_{\X_0}f(h_{\lambda}) w(m_0) \Omega_{0}^n,\] and this will prove both of the claims. The idea is these are both ``equivariant cohomological" integrals, only there are three obstructions: the function \(f\) may only be continuous rather than smooth, the variety \(\X_0\) may not be smooth, and the form \(\Omega_\Psi\) may also be not smooth. We can deal with these problems in a similar manner as the proof of \cite[Theorem 3.14]{Der18}.
	
	First, we assume \(\X_0\) and \(f\) are smooth. By Lemma \ref{lem:independence}, the integral is independent of the choice of smooth representative of the equivariant cohomology class \([\Omega_0+m_0]\). We may approximate \(\Omega_{\Psi,0}\) by a sequence of smooth K{\"a}hler metrics \(\Omega_{\Psi^\epsilon,0}\) with moment maps \(m_T^\epsilon\). Sending \(\epsilon\to0\), the integrals \(\int_{\X_0}f(h_{\Psi^\epsilon})w(m_{T,\Psi^\epsilon})\Omega_{\Psi^\epsilon, 0}^n\) are independent of \(\epsilon\) and converge to the limiting integral \(\int_{\X_0}f(h_{\Psi})w(m_{T,\Psi})\Omega_{\Psi, 0}^n\).
	
	Supposing now \(\X_0\) were smooth, we can deal with \(f\) not being differentiable by simply approximating it by smooth functions, then applying Lemma \ref{lem:independence} to the \(T\times S^1\)-action on \(\X_0\). Finally to deal with \(\X_0\) not being smooth, we can do exactly as in \cite[Theorem 3.14]{Der18}, namely we find a resolution of singularities \(\mathcal{Y}\to \X\) such the central fibre \(\mathcal{Y}_0\) is a simple normal crossings divisor, then compute the integrals on the resolution. Decomposing \(\mathcal{Y}_0\) into its smooth components and integrating over these, we see that the integrals are indeed equal.
\end{proof}
	
\begin{lemma}
	Let \((\X,\A_T)\) be a \(T\)-equivariant test configuration for \((X,\alpha_T)\) with \(\|(\X,\A_T)\|^w_1>0\). Normalise the test configuration so that \[(g(\A_T))(\xi) = 0.\] There exists a point \(p\in X\) (not necessarily \(T\)-invariant) such that \[\int_C\Omega \leq -\frac{1}{3V}\|(\X,\A)\|^w_1,\] where \(C\) is the orbit closure of \(p\) in \(\X\) under the \(\mathbb{C}^*\)-action, and \(V\) is the volume of \((X,\alpha_T)\) with respect to the measure \(w(\mu)\omega^n\).
\end{lemma}

\begin{proof}
	Denote by \(\psi_t\) the geodesic associated to the test configuration. Write \(\omega_t := \omega + i\d\db\psi_t\), which has associated moment map \(\mu_t := \mu + d^c \psi_t\). From Lemma \ref{lem:integral_limits}, \[\lim_{t\to\infty}\int_X |\dot{\psi}_t| w(\mu_t) \omega_t^n = \|(\X,\A)\|_1^w,\] and \[\lim_{t\to\infty} \int_X \dot{\psi}_t w(\mu_t) \omega_t^n = 0,\] where this last equality follows from the normalisation \((g(\A_T))(\xi) = 0\). Let \(\dot{\psi}_t = \dot{\psi}_t^+ + \dot{\psi}_t^-\) be the decomposition of \(\dot{\psi}_t\) into its positive and negative parts. Then \[\int_X |\dot{\psi}_t| w(\mu_t) \omega_t^n = \int_X \dot{\psi}_t^+ w(\mu_t) \omega_t^n - \int_X \dot{\psi}_t^- w(\mu_t) \omega_t^n\] and \[\int_X \dot{\psi}_t w(\mu_t) \omega_t^n = \int_X \dot{\psi}_t^+ w(\mu_t) \omega_t^n + \int_X \dot{\psi}_t^- w(\mu_t) \omega_t^n.\] Subtracting the first equation from the second and sending \(t\to\infty\), it follows that \begin{equation}\label{eq2}
		\lim_{t\to\infty}\int_X \dot{\psi}_t^- w(\mu_t) \omega_t^n = -\frac{1}{2}\|(\X,\A)\|_1^w.
	\end{equation} 
	
	Now, note that \(w(\mu_t) \omega_t^n\) is a non-negative Borel measure with fixed volume \(V := \int_X w(\mu) \omega^n\). If \(\dot{\psi}_t^-\) everywhere satisfied \[\dot{\psi}_t^- \geq -\frac{1}{3V}\|(\X,\A)\|_1^w,\] then we would have \[\int_X \dot{\psi}_t^- w(\mu_t) \omega_t^n \geq -\frac{1}{3}\|(\X,\A)\|_1^w,\] contradicting the above equality \eqref{eq2}. It follows that for each \(t\) there exists a point \(p_t\) such that \[\dot{\psi}_t(p_t) \leq -\frac{1}{3V}\|(\X,\A)\|_1^w.\] We wish to choose \(p_t\) independent of \(t\). To do this, we use the property of convexity along geodesics, which implies that \(\dot{\psi}_t(p)\) is non-decreasing in \(t\), for each fixed \(p\). Suppose that for every \(p\) there existed some \(T_p\) such that \[\dot{\psi}_t(p) \geq C := -\frac{1}{3V}\|(\X,\A)\|_1^w\] for all \(t \geq T_p\). Then the functions \(\varphi_t := \min(\dot{\psi}_t,C)\) are continuous and increase pointwise to the constant function \(C\), hence \(\varphi_t \to C\) uniformly. Thus for any \(\epsilon > 0\) there exists \(T\) such that \(\dot{\psi}_t(p) \geq C - \epsilon\) for all \(p\in X\) and \(t \geq T\), again contradiction equation \eqref{eq2}.
	
	It follows there exists \(p\in X\) such that \[\lim_{t\to\infty}\dot{\psi}_t^-(p) \leq - \frac{1}{3V}\|(\X,\A)\|_1^w,\] but the left hand side equals \(\int_C\Omega\) by the one-dimensional case of \cite[Theorem 4.14]{DR17} (which is also the one-dimensional and unweighted case of Lemma \ref{lem:integral_limits}) and we are done.
\end{proof}

The point \(p\) we constructed in the previous lemma may not be torus invariant, however a simple intersection theoretic argument by Dervan implies that there exists such a point which is \(T\)-invariant; we refer to  \cite[Proposition 4.14]{Der18} for the proof.\footnote{The proof of Proposition 4.14 has an error, although the alternative argument given after the proof is correct.}

\begin{lemma}
	Let \((\X,\A_T)\) be a \(T\)-equivariant test configuration for \((X,\alpha_T)\) with \(\|(\X,\A_T)\|_1^w > 0\), normalised so that \((g(\A_T))(\xi) = 0\). Then there exists a \(T\)-invariant point \(p\in X\) such that \[\int_C \Omega \leq -\frac{1}{3V}\|(\X,\A)\|_1^w,\] where \(V\) is the volume of \((X,\alpha)\).
\end{lemma}

We recall that the \(T\)-orthogonal weighted Chow weight of a \(T\)-invariant point \(p\) is \[\mathrm{Ch}_p^T(\X,\A_T):=\Ch_p^w(\X,\A)+\sum_{i=1}^r\langle\lambda,\beta_i\rangle_{(\X,\A_T)}(h_{\beta_i}(p)-\hat{h}_{\beta_i}).\] Here \(\lambda\) is the \(\mathbb{C}^*\)-action of the test configuration, \(\beta_i\) is an orthonormal basis for the Lie algebra \(\mathfrak{t}\) of \(T\), and \(h_{\beta_i}\) is a hamiltonian function on \(X\) for \(\beta_i\) with average \(\hat{h}_{\beta_i}\).

\begin{lemma}\label{lem:twisted_Chow_weight}
	Let \((\X,\A_T)\) be a test configuration for \((X,\alpha_T)\), and let \(\beta:\mathbb{C}^*\to T^{\mathbb{C}}\) be a one-parameter subgroup. Denote by \((\X,\A_T,\beta)\) the twist of the original test configuration by \(\beta\). Then given a \(T\)-invariant point \(p\in X\), the weighted Chow weight of \(p\) changes by \[\Ch_p^w(\X,\A_T,\beta) = \Ch_p^w(\X,\A_T) - h_\beta(p) + \hat{h}_\beta.\]
\end{lemma}

\begin{proof}
This is an easy consequence of a localisation formula already computed in Lemma \ref{lem:integral_twist}. In particular, we observed in that proof that for \(\A_T'\) the K{\"a}hler class of the twisted test configuration, that \[(g(\A_T'))(\xi) = (g(\A_T))(\xi) + \frac{1}{n!}\int_X h_\beta w(\mu)\omega^n.\] Hence \[\frac{(g(\A_T'))(\xi)}{(g'(\alpha_T))(\xi)} = \frac{(g(\A_T))(\xi)}{(g'(\alpha_T))(\xi)} + \hat{h}_\beta.\] Using this exact same formula in the one-dimensional case, we observe that \[\int_{C'}\Omega' = \int_C\Omega + h_\beta(p). \qedhere\]
\end{proof}

\begin{remark}
	It follows easily from this result that \(\Ch^T_p(\X,\A_T)\) is twist invariant.
\end{remark}

Using all of this, we can now prove the existence of a \(T\)-invariant point \(p\) with \(\Ch_p^T(\X,\A_T) > 0\) whenever \(\|(\X,\A_T)^\perp\|_1^w > 0\).

\begin{proof}[Proof of Proposition \ref{prop:Chow}]
	Let \[\beta := \sum_{i=1}^r \langle\lambda, \beta_i\rangle_{(\X,\A_T)} \beta_i\] be the orthogonal projection of \(\lambda\) to \(\mathfrak{t}\). Choose a sequence of rational points \(\beta_n\) in the Lie algebra tending towards \(\beta\). For each \(n\) we twist the test configuration by \(-\beta_n\), and there exists a \(T\)-invariant point \(p_n\) such that \[\Ch_{p_n}^w(\X,\A_T,-\beta_n) \leq -\frac{1}{3V}\|(\X,\A_T,-\beta_n)\|_1^w.\] The right hand side tends towards \[-\frac{1}{3V}\|(\X,\A_T)^\perp\|_1^w\] as \(n\to\infty\). On the other hand, by Lemma \ref{lem:twisted_Chow_weight} the left hand side satisfies \begin{align*}
	\Ch_{p_n}^w(\X,\A_T,-\beta_n) - \Ch_{p_n}^T(\X,\A_T) &= h_{\beta_n}(p) - \hat{h}_{\beta_n} - \sum_{i=1}^r\langle\lambda,\beta_i\rangle_{(\X,\A_T)}( h_{\beta_i}(p) - \hat{h}_{\beta_i}) \\
	&= h_{\beta_n}(p) - \hat{h}_{\beta_n} - (h_{\beta}(p) - \hat{h}_{\beta}) \\
	&\to 0
	\end{align*} as \(n\to\infty\) since \(\beta_n \to \beta\) as \(n\to\infty\). It follows that for \(n\) sufficiently large, \[\Ch_{p_n}^T(\X,\A_T)\leq -\frac{1}{6V}\|(\X,\A_T)^\perp\|_1^w,\] and we are done.
\end{proof}
	
\bibliographystyle{alpha}
\bibliography{references}
	
\end{document}